\documentclass{amsart}

\usepackage[dvipsnames]{xcolor}
\usepackage[margin=1.4in]{geometry} 
\usepackage{amsmath, amsfonts, amsthm, amstext, amssymb, mathtools, float} 
\usepackage[colorlinks=true,linkcolor=blue,citecolor=blue]{hyperref} 
\usepackage{todonotes} 
\usepackage{zref-clever} 
\zcsetup{nameinlink,cap}
\usepackage{tikz-cd, tikz} 
\tikzset{
  commutative diagrams/.cd, 
  arrow style=tikz, 
  diagrams={>=stealth}
}
\usepackage[shortlabels]{enumitem} 
\usepackage[mathscr]{euscript} 
\usepackage{stmaryrd} 
\usepackage{adjustbox}

\usepackage{wasysym}

\usetikzlibrary{shapes.geometric, arrows.meta, calc, positioning, svg.path}

\usepackage[shortalphabetic]{amsrefs}

\newcommand{\cref}[1]{\zcref[S]{#1}}
\newcommand{\Cref}[1]{\zcref[S]{#1}}

\zcRefTypeSetup{equation}{
  Name-sg-ab =,
  name-sg-ab =,
  Name-pl-ab =,
  name-pl-ab =,
  abbrev,
}

\zcRefTypeSetup{conjecture}{
Name-sg = Conjecture ,
name-sg = conjecture ,
Name-pl = Conjectures ,
name-pl = conjectures ,
}

\numberwithin{equation}{section} 
\numberwithin{table}{section} 
\numberwithin{figure}{section}

\theoremstyle{definition}
\AddToHook{env/definition/begin}{\zcsetup{countertype={equation=definition}}}
\newtheorem{definition}[equation]{Definition}
\AddToHook{env/defn/begin}{\zcsetup{countertype={equation=definition}}}
\newtheorem{defn}[equation]{Definition}
\AddToHook{env/remark/begin}
{\zcsetup{countertype={equation=remark}}}
\newtheorem{remark}[equation]{Remark}
\AddToHook{env/example/begin}{\zcsetup{countertype={equation=example}}}
\newtheorem{example}[equation]{Example}
\theoremstyle{plain}
\AddToHook{env/lemma/begin}{\zcsetup{countertype={equation=lemma}}}
\newtheorem{lemma}[equation]{Lemma}
\AddToHook{env/proposition/begin}{\zcsetup{countertype={equation=proposition}}}
\newtheorem{proposition}[equation]{Proposition}
\AddToHook{env/prop/begin}{\zcsetup{countertype={equation=proposition}}}
\newtheorem{prop}[equation]{Proposition}
\AddToHook{env/theorem/begin}{\zcsetup{countertype={equation=theorem}}}
\newtheorem{theorem}[equation]{Theorem}
\AddToHook{env/corollary/begin}{\zcsetup{countertype={equation=corollary}}}
\newtheorem{corollary}[equation]{Corollary}
\AddToHook{env/cor/begin}{\zcsetup{countertype={equation=corollary}}}
\newtheorem{cor}[equation]{Corollary}

\AddToHook{env/maintheorem/begin}{\zcsetup{countertype={equation=maintheorem}}}
\newtheorem{maintheorem}{Theorem} 
 
\zcRefTypeSetup{maintheorem}{Name-sg = Theorem, name-sg = theorem, Name-pl = Theorems, name-pl = theorems}

\AddToHook{env/conjecture/begin}{\zcsetup{countertype={equation=conjecture}}}

\zcRefTypeSetup{conjecture}{Name-sg = Conjecture, name-sg = conjecture, Name-pl = Conjectures, name-pl = conjectures}

\newenvironment{pf}{\begin{proof}}{\end{proof}}

\newcommand{\iso}{\cong}

\newcommand{\into}{\hookrightarrow}
\newcommand{\onto}{\twoheadrightarrow}

\newcommand{\ul}[1]{\underline{#1}} 

\newcommand{\cA}{\mathcal{A}} 
\newcommand{\uA}{\underline{A}} 
\newcommand{\K}{\mathcal{K}} 
\newcommand{\bbF}{\mathbb{F}} 
\newcommand{\F}{\bbF} 
\newcommand{\bH}{\mathbf{H}} 
\newcommand{\uI}{\underline{I}} 
\newcommand{\cP}{\mathcal{P}} 
\newcommand{\bbR}{\mathbb{R}} 
\newcommand{\R}{\bbR} 
\newcommand{\uT}{\underline{T}} 
\newcommand{\bbZ}{\mathbb{Z}} 
\newcommand{\Z}{\bbZ} 
\newcommand{\ulZ}{\ul{\Z}}
\newcommand{\ulF}{\ul{\F_2}}
\newcommand{\ulf}{\ul{f}}
\newcommand{\ulg}{\ul{g}}

\newcommand{\upi}{\ul{\pi}}

\newcommand{\genUL}{b_L}
\newcommand{\genUR}{b_R}

\newcommand{\genVR}{c}

\newcommand{\kfont}{\mathscr} 
\newcommand{\CH}{\bH} 
\newcommand{\KH}{\boldsymbol{\kfont{H}}} 
\newcommand{\KHN}{\boldsymbol{\kfont{HN}}} 
\newcommand{\KHND}{\boldsymbol{\kfont{HN}}_{\!D}} 

\newcommand{\CN}{\mathbf{N}} 
\newcommand{\KN}{\boldsymbol{\kfont{N}}} 
\newcommand{\KND}{\boldsymbol{\kfont{N}}_{\!D}} 

\DeclareMathOperator{\res}{res} 
\DeclareMathOperator{\tr}{tr} 
\DeclareMathOperator{\nm}{nm} 
\DeclareMathOperator{\RO}{RO}
\DeclareMathOperator{\BP}{BP}
\DeclareMathOperator{\eH}{H}

\newcommand{\orho}{\overline{\rho}} 

\newcommand{\Ab}{{\mathscr{A}\textnormal{b}}} 
\newcommand{\Epi}{{\mathscr{E}\textnormal{pi}}} 
\newcommand{\Fin}{{\mathscr{F}\textnormal{in}}} 
\newcommand{\Mack}{{\mathscr{M}\textnormal{ack}}} 
\newcommand{\Set}{{\mathscr{S}\textnormal{et}}} 
\newcommand{\Tamb}{{\mathscr{T}\textnormal{amb}}} 
\newcommand{\Sp}{\mathscr{S}\textnormal{p}} 

\newcommand{\mymidsizematrix}[1]
{\scalebox{0.7}{$\begin{pmatrix} #1 \end{pmatrix}$}}

\newcommand{\inductioncolor}{orange}
\newcommand{\normcolor}{blue}

\definecolor{acolor}{RGB}{209,187,215}

\definecolor{Ltextcolor}{RGB}{25, 101, 176}
\definecolor{Dtextcolor}{RGB}{212, 205, 55}
\definecolor{Rtextcolor}{RGB}{220, 5, 12}
\newcommand{\ucolor}{gray}
\newcommand{\xL}{{\color{Ltextcolor}{\mathbf{x}_L}}}
\newcommand{\xD}{{\color{Dtextcolor}{\mathbf{y}_D}}}
\newcommand{\xR}{{\color{Rtextcolor}{\mathbf{z}_R}}}
\newcommand{\vL}{{\color{Ltextcolor}{\mathbf{v}_L}}}
\newcommand{\vR}{{\color{Rtextcolor}{\mathbf{w}_R}}}
\newcommand{\uu}{{\color{\ucolor}{\mathbf{u}}}}

\def\abox{
     \begin{tikzpicture}
        \node at (0,0) [rectangle,draw,fill={gray!50},inner sep=0.75pt] {$\ast$};
     \end{tikzpicture}
}

\def\apent{
    \begin{tikzpicture}
      \node[regular polygon, fill={gray!50}, draw, regular polygon sides=5, 
 minimum width=0pt, inner sep = 0.5ex,] at (0,0) {};
      \node at (0,0) {$\ast$};
    \end{tikzpicture}
}

\newcommand{\phiLDRf}{
\begin{tikzpicture}[scale=0.6]
\draw[fill=black] (0,0) rectangle (0.5,0.15);
\end{tikzpicture}
}

\newcommand{\fillsquare}{
\begin{tikzpicture}
\node[draw, fill=black, inner sep=1pt, regular polygon sides=4] at (0,0) {$\ast$};
\end{tikzpicture}
}

\newcommand{\fillsquaredual}{
\begin{tikzpicture}
\node[draw, fill=black, inner sep=1pt, regular polygon sides=4, text=white!100] at (0,0) {\scriptsize $\ast$};
\end{tikzpicture}
}

\newcommand{\nH}{
\begin{tikzpicture}[scale=0.075]
	\draw[fill] (0,0) arc (0:180:2 and 1) arc (180:0:2 and 2.25);
\end{tikzpicture}
}

\newcommand{\nHdual}{
\begin{tikzpicture}[scale=0.075]
	\draw[fill] (0,0) arc (0:180:2 and 1) arc (180:0:2 and 2.25);
	\node[text=white] at (-2,1.5) {$*$};
\end{tikzpicture}
}

\newcommand{\ssfrac}[2]{
	\mathchoice{\raisebox{5pt}{$#1$}\!\big/\!\raisebox{-5pt}{$#2$}}
		{{}^{#1}\!/_{\!#2}}{ #1 /  #2}{ #1 /  #2}
	} 
\newcommand{\sfrac}[2]{{}^{#1}\!/_{\!#2}} 
\newcommand{\bigsfrac}[2]{\displaystyle \ssfrac{#1}{#2}} 

\begin{document}

\title{The $RO(\K)$-graded homotopy of Klein-four normed Mackey functors}
\author{Bertrand Guillou}
\email{bertguillou@uky.edu}
\author{Jesse Keyes}
\email{jdke228@uky.edu}
\author{David Mehrle}
\email{davidm@uky.edu}
\address{Department of Mathematics, University of Kentucky, Lexington, KY, U.S.A.}


\begin{abstract}
We compute the $RO(\K)$-graded coefficients of the equivariant Eilenberg--Mac~Lane spectrum associated to various Hill--Hopkins--Ravenel norms of the constant-$\bbF_2$ Mackey functor, where $\K$ is the Klein-four group. Further, we analyze the multiplicative structure of these $\RO(\K)$-graded Tambara functors.
\end{abstract}

\maketitle


\setcounter{tocdepth}{1}
\tableofcontents

\section{Introduction}

In equivariant stable homotopy theory, ordinary cohomology is represented
by an equivariant Eilenberg--Mac~Lane $G$-spectrum $H\underline{M}$, where $\underline{M}$ is a Mackey functor for the group $G$. One may expect that the coefficients of an equivariant Eilenberg--Mac~Lane spectrum are easy to understand, but this is more complicated than in the non-equivariant setting. The homotopy of a $G$-spectrum $E$ can be considered as $\RO(G)$-graded, where $\RO(G)$ is the real representation ring of $G$. In the case that $E=H\underline{M}$ for some $G$-Mackey functor $\underline{M}$, the $\RO(G)$-graded coefficients of $E$ correspond to the Bredon homology of virtual real representation spheres of $G$ with coefficients in $\underline{M}$. 
	
For $G=C_2$, the $RO(G)$-graded coefficients of any equivariant Eilenberg--Mac~Lane spectrum are quite well-understood \cite{Sikora}. This is far from true for $G = \K = C_2 \times C_2$, the Klein-four group. One computational difficulty that arises in this context is that $\RO(\K)$ is a free abelian group of rank four. Despite this, some computations of the $\RO(\K)$-graded coefficients of $H\underline{M}$ have been done \cites{HoKr,GY,EB,Slone,JK}. In this paper, we make further contributions to these computations. We compute a portion of the $\RO(\K)$-graded homotopy Mackey functors of $\underline{M} = N_H^\K \underline{\F_2}$, where $H$ is a proper subgroup of $\K$ and $\underline{\F_2}$ the constant $H$-Mackey functor at $\F_2$. Here, $N_H^\K(-): \Mack_H \rightarrow \Mack_\K$ is the Mackey functor norm.
We also introduce a method for calculating the Mackey functors $N_H^G(\ulF)$
for any group $G$ and subgroup $H$, based on Tambara ideals of the Burnside Tambara functor.

Hill, Hopkins, and Ravenel introduced the norm functor $N_H^G(-): \Sp^H \rightarrow \Sp^G$ in \cite{Kervaire} and have studied various norms $N_{C_2}^G \BP_\R$, where $\BP_\R$ is the Real Brown-Peterson $C_2$-spectrum. 
Despite the  success of this analysis, many computations concerning equivariant spectra related to this norm construction remain mysterious. Along these lines, \cite{MSZ}*{Theorem~4.4} gives partial information about the homotopy of the geometric fixed points $\Phi^{C_2}N_{C_2}^{C_4} \BP_\R\simeq N_e^{C_2} H{\F_2}$. The first few 
{$\Z$-graded} 
homotopy Mackey functors are listed in \cref{tbl:HtpyNHF} 
{using notation as indicated in \cref{tab-C2Mackey},}
 and the $RO(C_2)$-graded homotopy groups $\pi_{x+y\sigma} N_e^{C_2} H\bbF_2$ are displayed for \mbox{$x\leq 6$} and $x+y\leq 6$
in \cref{fig:piC2NHF2}. 
See the beginning of \cref{sec:charts} for a discussion of \cref{fig:piC2NHF2}.
  A complete calculation of the homotopy of $N_{C_2}^{C_4}\BP_\R$ and $N_e^{C_2} H \F_2$ is out of reach. 
 {Part of the reason for this is that the underlying homotopy groups of $N_e^{C_2} H \F_2$, together with the $C_2$-action, form the dual Steenrod algebra with the action of the antipode. There is no known formula for the fixed points of this action. }
 In contrast, the $\RO(C_2)$-graded homotopy groups of both $H_{C_2} N_e^{C_2} \F_2$, the $0^{th}$-Postnikov truncation of $N_e^{C_2} H\F_2$, and $H_{C_2} \underline{\F_2}$ are completely understood \cites{Dugger,Sikora}.
The $RO(C_2)$-graded homotopy groups of the Eilenberg--Mac~Lane spectra $H_{C_2} N_e^{C_2} \bbF_2$ and $H_{C_2} \ul\bbF_2$ are displayed in \cref{fig:piC2NF2} and \cref{fig:piC2F2}, respectively.

Motivation for the study of $N_{C_2}^{C_4}\BP_\R$ comes from chromatic homotopy theory \cite{HSWX}.
The quaternion group $Q_8$ also plays an important role in chromatic homotopy theory, which suggests the study of $N_{C_2}^{Q_8} \BP_\R$, where $C_2$ is the center of $Q_8$.
This is expected to be difficult; however, given that the quotient $Q_8/C_2$ is isomorphic to $\K$, one may wish to compute $\Phi^{C_2}N_{C_2}^{Q_8} \BP_\mathbb{R} \simeq N_e^{\K} H\F_2$. Again, this is out of reach, though its Postnikov truncation $H_\K N_e^\K \F_2$ can be completely computed. The $RO(\K)$-graded homotopy groups of $H\underline{\F_2}$ were previously computed in \cites{HoKr,EB}, while some of the homotopy Mackey functors were determined in \cite{GY}. These are depicted in a range in \cref{fig:HF}. In this article, we compute a portion of the $\RO(\K)$-graded homotopy Mackey functors of $H_\K N_e^{\K} {\F_2}$. It is common to use the symbol $\bigstar$ to denote a grading over $RO(G)$. We will use the symbol $\blacklozenge$ to denote a grading over the $\mathrm{Aut}(\K)$-fixed subgroup $\Z\{1,\orho\}\subset RO(\K)$, where $\orho$ is the reduced regular representation.

\begin{maintheorem}
\label{thmA}
The homotopy Mackey functors $\upi_{\blacklozenge}H_\K N_e^\K \F_2$ are as described in \cref{sec:HNeF} and displayed in \cref{fig:HN,fig:HNmult}, where $\blacklozenge \in \Z\{1,\orho\}\subset RO(\K)$.

In particular, the quotient map $ N_e^\K \F_2 \rightarrow \ulF$ induces an isomorphism 
\[ \upi_{n+k\orho}(HN_e^\K \F_2) \rightarrow \upi_{n+k\orho}(H\ulF) \]
in a subrange of each of the positive and negative cones. These subranges are above the line through $(3,-1)$ with slope $-1$ and below the line through $(-2,1)$ with slope $-1$, respectively. 
\end{maintheorem}

We also consider the intermediate norm $N_C^\K \ulF$, where $C\leq K$ is an order 2 subgroup. 
The corresponding Eilenberg--Mac~Lane spectrum $H_\K N_C^\K \ulF$ is the Postnikov truncation of the normed spectrum $N_C^\K H_C \ulF$. The latter is a useful intermediary between the mysterious $N_e^\K H\F_2$ and
{the well-understood}
 $H_\K \ulF$, in that $H_\K\ulF = N_\K^\K H_\K\ulF$.
For definiteness, we specialize to the case that $C=D$ is the diagonal subgroup of $\K$, though the other choices can be obtained by using the $\mathrm{Aut}(\K)$-action. 

\begin{maintheorem}
\label{thmB}
The homotopy Mackey functors $\upi_{\blacklozenge}H_\K N_D^\K \ulF$ are as described in \cref{sec:HNDF} and displayed in \cref{fig:HND,fig:HNDmult}, where $\blacklozenge \in \Z\{1,\orho\}\subset RO(\K)$. 

As in \cref{thmA}, the quotient map $ N_D^\K \ulF \rightarrow \ulF$ induces an isomorphism 
\[ \upi_{n+k\orho}(HN_D^\K \ulF) \rightarrow \upi_{n+k\orho}(H\ulF) \]
in a subrange of each of the positive and negative cones. These subranges are above the line through $(3,-1)$ with slope $-1$ and below the line through $(-2,1)$ with slope $-1$, respectively.
\end{maintheorem}

\begin{table}
\caption{The homotopy Mackey functors $\underline\pi_n N_e^{C_2} H \bbF_2$, $n\leq 6$. See \cref{tab-C2Mackey} for the Mackey functor Lewis diagrams. }
\label{tbl:HtpyNHF}
$\begin{array}{c|ccccccccc}
n & 0 & 1 & 2 & 3 & 4 & 5 & 6 
\\
\hline
\ul\pi_n N_e^{C_2} H \bbF_2 
& N_e^{C_2} \bbF_2 
&   
\ulg \oplus \ulf
& 
N_e^{C_2} \bbF_2 
&  
\ulg\, \oplus
\uparrow_e^{C_2}\! \bbF_2
& 
\uparrow_e^{C_2}\! \bbF_2
& 
\uparrow_e^{C_2}\! \bbF_2
& 
N_e^{C_2} \bbF_2 \oplus  
\uparrow_e^{C_2}\! \bbF_2
&
\end{array}$
\end{table}

\subsection{Conventions}

We write $e$ for a trivial group and $C_2 := \langle \gamma \mid \gamma^2 = 1 \rangle$ for a finite group of order two. Our main group of interest is the Klein four-group $\K := C_2 \times C_2$; its nontrivial subgroups are $L := C_2 \times e$, $D := \langle (\gamma, \gamma) \rangle$, and $R := e \times C_2$. 

We write both $\Z/2$ and $\F_2$ for the ring of order 2, guided by aesthetics.

We use different fonts to differentiate between non-equivariant, $C_2$-equivariant, and $\K$-equivariant homotopy theory. We write $\eH$ for non-equivariant Eilenberg--Mac~Lane spectra, $\CH$ for $C_2$-equivariant Eilenberg--Mac~Lane spectra, and $\KH$ for $\K$-equivariant Eilenberg--Mac~Lane spectra. 
Similarly, we will often abbreviate the Mackey functors {$N_e^{C_2} \mathbb{F}_2, N_e^\K \bbF_2,$ and $ N_D^\K \ulF$ by 
$\CN, \KN,$ and $\KND$, respectively}.

Equivariant spectra will always be considered as indexed over a complete universe, so that their homotopy is valued in Mackey functors.
Our calculations require many $C_2$- or $\K$-Mackey functors; notation and definition for these can be found in \cref{tab-C2Mackey} and  \cref{tab-K4Mackey}.

\subsection{Acknowledgements}

We thank Mike Hill, Danny Shi, and Guoqi Yan for helpful discussions and Anna Marie Bohmann for guidance on visualization. We also thank an anonymous referee for their helpful suggestions.
This work was supported by NSF grants DMS-2403798 and DMS-2135884 and  Simons Foundation award MPS-TSM-00007067.

\subsection{Competing interests}

The authors declare none.

\section{$\K$-Norms via Tambara Ideals}
\label{sec:algebra}

The Hill--Hopkins--Ravenel norm $N_H^G \colon \Sp^H \to \Sp^G$ is captured on the level of Mackey functors by a functor $N_H^{G} \colon \Mack(H) \to \Mack(G)$. As a non-additive functor between additive categories, the norm is often difficult to compute. We introduce a technique for computing $N_H^G(\ulF)$ based on Tambara ideals, and use it to compute $N_e^\K(\F_2)$ and $N_C^\K(\ulF)$, where $C$ is any of the order two subgroups of $\K$.

\subsection{Norms of Mackey and Tambara Functors}

Recall that the category of $G$-Mackey functors is the category of additive functors $\cA_G \to \Ab$, where $\cA_G$ is the Burnside category for $G$.

\begin{definition}[{\cite{Hoyer}*{Definition 2.3.2}}]
	The functor $N_H^G \colon \Mack(H) \to \Mack(G)$ is given by left Kan extension along coinduction $\Fin_H(G,-) \colon \cA_H \to \cA_G$. 
\end{definition}

Tambara functors are the $G$-commutative monoids in the category of $G$-Mackey functors.
Alternatively, the category of $G$-Tambara functors is the category of product preserving functors $\uT \colon \cP^G \to \Set$ such that each $\uT(U)$ is a commutative ring, where $\cP^G$ is the category of polynomials of finite $G$-sets \cite{Tam1993}*{Section 8}. Concretely, a Tambara functor is a Mackey functor valued in commutative rings whose restrictions are ring homomorphisms satisfying Frobenius reciprocity (a Green functor) equipped with norm maps (of multiplicative monoids) satisfying Tambara reciprocity.

The norm $n_H^G$ from $H$-Tambara functors to $G$-Tambara functors is slightly different from the norm of Mackey functors. 

\begin{definition}[{\cite{BH2018}*{Definition 6.8, Proposition 6.9}}]
	The functor $n_H^G \colon \Tamb(H) \to \Tamb(G)$ is given by left Kan extension along the inclusion $\cP^H \to \cP^G$; it is left adjoint to the restriction functor $\res^G_H \colon \Tamb(G) \to \Tamb(H)$. 
\end{definition}

We recall a theorem relating the two functors: 

\begin{theorem}[{\cite{Hoyer}*{Theorem 2.3.3}}]
	\label{HoyerMazurTheorem}
	The following square commutes, where vertical arrows are forgetful functors: 
	\[
	\begin{tikzcd}
		\Tamb(H) \ar[r, "n_H^G"] \ar[d] & \Tamb(G) \ar[d] \\
		\Mack(H) \ar[r, "N_H^G"] & \Mack(G)
	\end{tikzcd}
	\]
\end{theorem}

\subsection{Tambara Ideals and Norms}
Like commutative rings, Tambara functors have a robust theory of ideals. We use non-unital Tambara functors to define Tambara ideals. 

Let $\Epi^G \subseteq \Fin^G$ be the category of finite $G$-sets and surjections. Note that $\Epi^G$ is a pullback stable subcategory of finite $G$-sets, i.e. pullbacks in $\Fin^G$ of morphisms in $\Epi^G$ are again in $\Epi^G$. Let $\cP^G_\Epi$ be the category of polynomials with exponents in $\Epi^G$ \cite{BH2018}*{Definition 2.7}. 

\begin{definition}[{\cite{BH2018}*{Definition 4.15}}]
	A \emph{non-unital Tambara functor} is a product preserving functor $\uT \colon \cP^G_\Epi \to \Set$ such that each $\uT(X)$ is an abelian group. 
\end{definition}

Concretely, a non-unital Tambara functor is a Tambara functor valued in non-unital rings, that is, a Mackey functor valued in non-unital commutative rings whose restrictions are ring homomorphisms satisfying Frobenius reciprocity (a non-unital Green functor) equipped with norm maps (of non-unital multiplicative monoids) satisfying Tambara reciprocity. 

The definition below is equivalent to the original definition given by Nakaoka \cite{ideals}*{Definition 2.1}.

\begin{definition}[{\cite{Hill2017}*{sentence before definition 5.1}}]
	A \emph{Tambara ideal} $\uI$ of a Tambara functor $\uT$ is a sub-non-unital Tambara functor of $\uT$ with a morphism of non-unital Tambara functors $\uT \boxtimes \uI \to \uI$.
\end{definition}

Practically speaking, a Tambara ideal $\uI$ of a Tambara functor $\uT$ is a collection of ideals $\uI(G/H) \subseteq \uT(G/H)$ closed under restriction, transfer, and norm. 

\begin{example}
	Recall that the Burnside ring for $C_2$ is
\[
	A(C_2) \iso \bigsfrac{\Z[t]}{(t^2-2t),}
\]
where $t = [\sfrac{C_2}{e}]$.
Let $\uI$ be the Tambara ideal generated by $2 \in \uA(C_2/C_2)$ inside the $C_2$-Burnside Tambara functor $\uA$. 
	This is the smallest Tambara ideal of $\uA$ containing $2 \in \uA(C_2/C_2)$. This ideal must contain 
	\[
		2 = \res(2) \in \uA(C_2/e)
	\] at the underlying level and 
	\[
		2t = \tr\res(2),\quad 2 + t = \nm \res(2) \in \uA(C_2/C_2)
	\] 
	at the fixed level. Altogether, a minimal generating set for $\uI(C_2/C_2)$ is $(2,t)$ and a minimal generating set for $\uI(C_2/e)$ is $(2)$. The quotient Tambara functor $\uA/\uI$ is therefore isomorphic to $\underline{\F_2}$.
	\[
		\begin{tikzcd}[column sep=3cm]
			(2,t)
				\ar[d, "\res"] 
				&
			\Z[t] / (t^2 - 2t)
				\ar[d, "t \mapsto 2" description]
				&
			\F_2 
				\ar[d, "1" description]
				\\[1.5cm]
			(2) 
				\ar[u, bend right=50, orange, "\tr"']
				\ar[u, bend left=50, blue, "\nm"] 
				&
			\Z 
				\ar[u, bend right=50, orange, "a \mapsto at"']
				\ar[u, bend left=50, blue, "a \mapsto a + \frac{a^2 - a}{2} t"]
				&
			\F_2
				\ar[u, bend right=50, "2"', orange]
				\ar[u, bend left=50, blue, "a \mapsto a^2"]
				\\
			\uI 
				\ar[r, hook] 
				& 
			\uA
				\ar[r, two heads]
				&
			\underline{\F_2}
		\end{tikzcd}
	\]
\end{example}
		
	It is often the case that many Mackey functors of interest can be written as quotients of the Burnside Tambara functor. For instance, the previous example shows that $\underline{\F_2}$ is a quotient of $\uA$, and $\ulZ$ is a quotient of $\uA$ by the Tambara ideal generated by $t-2 \in \uA(C_2/C_2)$.
	Writing $\underline{\F_2}$ as the quotient of $\uA$ by this Tambara ideal is a productive strategy to compute its norms. 

\begin{proposition}
	\label{norms as quotients}
	Let $G$ be a group which contains a subgroup isomorphic to $C_2$.
	The norm $n_{C_2}^G(\ulF)$ is the quotient of the $G$-Burnside Tambara functor $\uA$ by the Tambara ideal generated by $2 \in \uA(G/{C_2})$. 
\end{proposition}

\begin{proof}
 	The ${C_2}$-Tambara ideal $\uI \subseteq \uA$ generated by $2 \in \uA({C_2}/{C_2})$ is the image of the ${C_2}$-Tambara functor homomorphism $\uA[x_{{C_2}/{C_2}}] \to \uA$ determined by $x \mapsto 2$, where $\uA[x_{{C_2}/{C_2}}]$ is the free $H$-Tambara functor generated at the top level \cite[Definition 5.4]{BH2018}. 
	
	We may therefore write the ${C_2}$-Tambara functor $\underline{\F_2}$ as a reflexive coequalizer in the category $\Tamb({C_2})$:
	\[
		\begin{tikzcd}
			\uA[x_{{C_2}/{C_2}}] 		
				\ar[r, shift left=2mm, "x \mapsto 2"]
				\ar[r, shift right=2mm, "x \mapsto 0"'] 
				&
			\uA 
				\ar[l, hook']
				\ar[r, two heads]
				&
			\underline{\F_2}
		\end{tikzcd}
	\]
	Since $n_{{C_2}}^G \colon \Tamb({C_2}) \to \Tamb(G)$ is a left adjoint, it preserves coequalizers. 
	The Tambara norm $n_{{C_2}}^G$ sends the ${C_2}$-Burnside functor to the $G$-Burnside functor. By \cite[Proposition 4.2]{HMQ}, we know that $n_{{C_2}}^G(\uA[x_{{C_2}/{C_2}}])$ is isomorphic to $\uA[x_{G/{C_2}}]$. Thus, we have a reflexive coequalizer in $\Tamb(G)$:
	\[
		\begin{tikzcd}
			\uA[x_{G/{C_2}}] 		
				\ar[r, shift left=2mm, "x \mapsto 2"]
				\ar[r, shift right=2mm, "x \mapsto 0"'] 
				&
			\uA 
				\ar[l, hook']
				\ar[r, two heads]
				&
			n_{{C_2}}^G(\underline{\F_2})
		\end{tikzcd}
	\]
	This expresses $n_{{C_2}}^G(\underline{\F_2})$ as the quotient of the $G$-Tambara functor $\uA$ by the ideal generated by $2 \in \uA(G/{C_2})$.
\end{proof}

{Although the proposition above is stated for the norm from $C_2$ to $G$, a similar strategy works to calculate other norms, as the following example shows. }

\begin{example}
	\label{norm from e to C2 of F2}
	$n_e^{C_2}(\F_2) \cong \uA/\underline{J}$, where $\underline{J}$ is the Tambara ideal of $\uA$ generated by $2 \in \uA(C_2/e)$. At the fixed level, minimal generators for this ideal are $(2t, 2+t)$. 
	\[
		\begin{tikzcd}[column sep=3cm]
			(2t,2+t)
				\ar[d, "\res"] 
				&
			\bigsfrac{\Z[t]}{(t^2 - 2t)}
				\ar[d, "t \mapsto 2" description]
				&
			\Z/4
				\ar[d, "1" description]
				\\[1.5cm]
			(2) 
				\ar[u, bend right=50, \inductioncolor, "\tr"']
				\ar[u, bend left=50, \normcolor, "\nm"] 
				&
			\Z 
				\ar[u, bend right=50, \inductioncolor, "a \mapsto at"']
				\ar[u, bend left=50, \normcolor, "a \mapsto a + \frac{a^2 - a}{2} t"]
				&
			\Z/2
				\ar[u, bend right=50, "2"', \inductioncolor]
				\ar[u, bend left=50, \normcolor, "a \mapsto a^2"]
				\\
			\underline{J}
				\ar[r, hook] 
				& 
			\uA
				\ar[r, two heads]
				&
			n_e^{C_2}(\F_2)
		\end{tikzcd}
	\]
\end{example}

\begin{remark}
	If $\uI$ is a Tambara ideal of a Tambara functor $\uT$, we might expect a statement like the following to be true: $N_H^G(\uI)$ is a Tambara ideal of $n_H^G(\uT)$, and the norm of the quotient $n_H^G(\uT/\uI)$ is the quotient of the norm $n_H^G(\uT) / N_H^G(\uI)$. There are several obstacles to making this statement precise. Although $N_H^G \colon \Mack(H) \to \Mack(G)$ is a left adjoint, it is not an exact functor (nor even additive). Furthermore, $\uT/\uI$ is not a colimit in $H$-Tambara functors, and there's no reason it should be preserved by $n_H^G$.
\end{remark}

\subsection{The Norm \texorpdfstring{$N_e^\K(\F_2)$}{from e to K of F_2}}
\label{sec:NeKF}

Recall that the Burnside ring for $\K$ is
\[
	A(\K) \iso \bigsfrac{\Z[t_L,t_D,t_R]}{t_Lt_D = t_Lt_R = t_Dt_R, t_\bullet^2 = 2t_\bullet,}
\]
where $t_L = [\sfrac{\K}{L}]$, $t_D = [\sfrac{\K}{D}]$, and $t_R = [\sfrac{\K}{R}]$. The relation $t_\bullet^2 = 2t_\bullet$ holds for all $\bullet \in \{L,D,R\}$. Note that the class $[\sfrac{\K}{e}]$ is unnecessary as a generator, since $[\sfrac{\K}{e}] = t_Lt_D = t_Lt_R = t_Dt_R$.
The Burnside Tambara functor $\uA$ for the Klein four-group is displayed in \cref{figure:BurnsideTambaraFunctor}.
In the Tambara functor $\uA$, the norm 
$\nm_e^L$ is given by 
\[
	\nm_e^L(a) = a + \left(\tfrac{a^2 - a}{2}\right) s_L,
\]
and the norm $\nm_L^\K$ is given by 
\[
\nm_L^\K(a\,s_L + b)\ =\ \left(a^2-a+ba \right)t_Dt_R \, + \, \left(\tfrac{b^2-b}{2}\right)t_L \, + \, a\,t_D \, + \, a\, t_R + b,
\]
and similarly for norms to or from $D$ and $R$.

\begin{figure}[h]
\caption{The Burnside $\K$-Tambara functor $\uA$.}
\label{figure:BurnsideTambaraFunctor}
\begin{center}
	\begin{tikzpicture}[scale=0.9]
		\node (Ae) at (0,-4) {$\Z$};
		\node (AL) at (-4,0) {$\bigsfrac{\Z[s_L]}{s_L^2 -2s_L}$};
		\node (AD) at (0,0) {$\bigsfrac{\Z[s_D]}{s_D^2 - 2s_D}$};
		\node (AR) at (4,0) {$\bigsfrac{\Z[s_R]}{s_R^2 - 2s_R}$};
		\node (AK) at (0,4) {$\bigsfrac{\Z[t_L,t_D,t_R]}{\big(t_Lt_D = t_Lt_R = t_Dt_R, t_\bullet^2 = 2t_\bullet\big)}$};
					
		\draw[-Stealth,\inductioncolor] (Ae) to[bend right=20] node[fill=white, near end]{\tiny $\cdot s_D$} (AL);
		\draw[-Stealth,\inductioncolor] (Ae) to[bend right=25] node[fill=white, near end]{\tiny $\cdot s_D$} (AD);
		\draw[-Stealth,\inductioncolor] (Ae) to[bend right=20] node[fill=white, near end]{\tiny $\cdot s_D$} (AR);
		
		\draw[-Stealth,\normcolor] (Ae) to[bend left=20] node[fill=white,near end]{\tiny$\nm_e^L$} (AL);	
		\draw[-Stealth,\normcolor] (Ae) to[bend left=25] node[fill=white,near end]{\tiny$\nm_e^D$} (AD);	
		\draw[-Stealth,\normcolor] (Ae) to[bend left=20] node[fill=white,near end]{\tiny$\nm_e^R$} (AR);					
		\draw[-Stealth] (AD) --node[fill=white,rotate=-90]{\tiny $s_D \mapsto 2$} (Ae);
		\draw[-Stealth] (AL) --node[fill=white,rotate=-45]{\tiny $s_D \mapsto 2$} (Ae);
		\draw[-Stealth] (AR) --node[fill=white,rotate=45]{\tiny $2 \mapsfrom s_D$} (Ae);	

		\draw[-Stealth,\inductioncolor] 
			(AD) 	to[bend right=20]
					node[fill=white,rotate=90]{\tiny$1 \mapsto t_D$}
					node[below,rotate=90]{\tiny$s_D \mapsto t_Lt_R$}
			(AK);
		\draw[-Stealth,\inductioncolor] 
			(AL) 	to[bend right=20]
					node[fill=white,rotate=45]{\tiny$1 \mapsto t_L$}
					node[below,rotate=45]{\tiny$s_L \mapsto t_Dt_R$}
			(AK);
		\draw[-Stealth,\inductioncolor] 
			(AR) 	to[bend right=20]
					node[fill=white,rotate=-45]{\tiny$t_Lt_D \mapsfrom s_R$}
					node[above,rotate=-45]{\tiny$t_R \mapsfrom 1$}					
			(AK);
		
		\draw[-Stealth,\normcolor] 
			(AD) 	to[bend left=20]
					node[very near start, left]{\tiny $\nm_D^K$}
			(AK);
		\draw[-Stealth,\normcolor] 
			(AL) 	to[bend left=20]
					node[pos=0.4,fill=white]{\tiny $\nm_L^K$}
			(AK);
		\draw[-Stealth,\normcolor] 
			(AR) 	to[bend left=15]
					node[fill=white,pos=0.19]{\tiny $\nm_R^K$}
			(AK);
			
		\draw[-Stealth] 
			(AK) 	-- 
					node[rotate=-90,fill=white]{\tiny $t_L,t_R \mapsto s_D$} 
					node[below,rotate=-90]{\tiny $t_D \mapsto 2$}
			(AD);
		\draw[-Stealth] 
			(AK) 	-- 
					node[rotate=-45,fill=white]{\tiny $t_L,t_D \mapsto s_R$} 
					node[below,rotate=-45]{\tiny $t_R \mapsto 2$}
			(AR);
		\draw[-Stealth] 
			(AK) 	-- 
					node[rotate=45,fill=white]{\tiny $s_L \mapsfrom t_D,t_R$} 
					node[below,rotate=45]{\tiny $2 \mapsfrom t_L$}
			(AL);
	\end{tikzpicture}	
\end{center}
\end{figure}

\begin{proposition}
\label{prop:Norm e to K top level}
	The value of the $\K$-Tambara functor $n_e^\K(\F_2)$ at the trivial orbit $\K/\K$ is the ring
	\[
		n_e^\K(\F_2)(\K/\K) \iso \bigsfrac{\Z/8[\genUL,\genUR]}{(2\genUL\genUR,2\genUR,\genUL^2,\genUR^2,\genUL\genUR)},
	\]
	where $\genUL$ and $\genUR$ are the images of $t_L+2$ and $t_R+2$, respectively, under the surjection $\ul{A} \onto n_e^\K(\Z/2)$. This surjection sends $t_D$ to $\genUL+\genUR+2$.
\end{proposition}

\begin{proof}
	We apply \cref{norm from e to C2 of F2}: $n_e^\K(\F_2)$ is the quotient of $\uA$ by the ideal generated by $2 \in \uA(\K/e)$. Thus, $n_e^\K(\F_2)(\K/\K)$ is the quotient of $\uA(\K/\K)$ by all norms and transfers of $2$. 
	In other words, we are asking for the quotient of $\Z[t_L, t_D, t_R]$ by the ideal generated by the relations in the Burnside ring $A(\K)$:
		\begin{align*}
			&t_Lt_D -t_Lt_R\\
			&t_Lt_R - t_Dt_R\\
			&t_L^2 - 2t_L\\
			&t_D^2 - 2t_D\\
			&t_R^2 - 2t_R
		\end{align*}
		and the relations imposed by the ideal of $\uA$ generated by $2 \in \uA(\K/e)$:
		\begin{align}
			\tr_e^\K(2) &= 2t_Lt_D \notag \\
			\nm_e^\K(2) &= 2t_Lt_D + t_L + t_D + t_R + 2. \label{top of ideal gen by 2 in A(K/e)} \notag \\
			\tr_L^\K(\nm_e^L(2)) &= 2 t_L + t_Dt_R \notag \\
			\nm_L^\K(\tr_e^L(2)) &= 2 t_D + 2 t_R + 2 t_D t_R \notag \\ 
			\tr_D^\K(\nm_e^D(2)) &= 2 t_D + t_Lt_R  \\
			\nm_D^\K(\tr_e^D(2)) &= 2 t_L + 2 t_R + 2 t_L t_R \notag \\ 
			\tr_R^\K(\nm_e^R(2)) &= 2 t_R + t_Lt_D \notag \\
			\nm_R^\K(\tr_e^R(2)) &= 2 t_L + 2 t_D + 2 t_L t_D. \notag 
		\end{align}
		By SAGE, a Gr\"obner basis for this ideal is:
		\[
			t_L^{2} + 4, \quad t_L t_R + 4, \quad t_R^{2} + 4, \quad t_L + t_D + t_R + 2, \quad 2 t_L + 4, \quad 2 t_R + 4, \quad 8.
		\]
		Let $I$ be the ideal generated by the above. Write $\bar t_L$, $\bar t_D$, and $\bar t_R$ for the images of $t_L, t_D,$ and $t_R$ in the quotient $\Z[t_L, t_D, t_R]/I$. Since $8 \in I$, the $\Z$ becomes a $\Z/8$. We may eliminate any one of $\bar t_L, \bar t_D,$ or $\bar t_R$ using the relation $t_L + t_D + t_R + 2$. Thus, we have 
		\[
			\bigsfrac{\Z[t_L,t_D,t_R]}{I} \cong \bigsfrac{\sfrac{\Z}{8}[\bar t_L, \bar t_R]}{(2\bar t_L+4, 2\bar t_R+4,\bar t_L^2+4,\bar t_R^2+4,\bar t_L\bar t_R+4)}.
		\]
		Setting $\genUL = \bar t_L + 2$ and $\genUR = \bar t_R + 2$, we have 
		\[
			\bigsfrac{\Z[t_L,t_D,t_R]}{I} \cong \bigsfrac{\sfrac{\Z}{8}[\genUL,\genUR]}{(2\genUL,2\genUR,\genUL^2,\genUR^2,\genUL\genUR)}. 
		\]
\end{proof}

\begin{figure}
\caption{The Tambara functor $n_e^\K(\F_2)$.}
\label{figure:NormOfZmod2}
	\begin{center}
	\begin{tikzpicture}[scale=0.9]
		\node (Ae) at (0,-4) {$\Z/2$};
		\node (AL) at (-4,0) {$\Z/4$};
		\node (AD) at (0,0) {$\Z/4$};
		\node (AR) at (4,0) {$\Z/4$};
		\node (AK) at (0,4) {$\bigsfrac{\sfrac{\Z}{8}[\genUL,\genUR]}{\big(2\genUL,2\genUR,\genUL^2,\genUR^2,\genUL\genUR\big)}$};
					
		\draw[-Stealth,\inductioncolor] (Ae) to[bend right=20] node[fill=white]{$2$} (AL);
		\draw[-Stealth,\inductioncolor] (Ae) to[bend right=25] node[fill=white]{$2$} (AD);
		\draw[-Stealth,\inductioncolor] (Ae) to[bend right=20] node[fill=white]{$2$} (AR);
		
		\draw[-Stealth,\normcolor] (Ae) to[bend left=20] node[fill=white]{$1$} (AL);	
		\draw[-Stealth,\normcolor] (Ae) to[bend left=25] node[fill=white]{$1$} (AD);	
		\draw[-Stealth,\normcolor] (Ae) to[bend left=20] node[fill=white]{$1$} (AR);					
		
		\draw[-{Stealth[] Stealth}] (AD) -- (Ae);
		\draw[-{Stealth[] Stealth}] (AL) -- (Ae);
		\draw[-{Stealth[] Stealth}] (AR) -- (Ae);	

		\draw[-Stealth,\inductioncolor] 
			(AD) 	to[bend right=20]
					node[fill=white,rotate=90]{\tiny$1 \mapsto \genUL+\genUR-2$}
			(AK);
		\draw[-Stealth,\inductioncolor] 
			(AL) 	to[bend right=20]
					node[fill=white,rotate=45]{\tiny$1 \mapsto \genUL-2$}
			(AK);
		\draw[-Stealth,\inductioncolor] 
			(AR) 	to[bend right=20]
					node[fill=white,rotate=-45]{\tiny$\genUR-2 \mapsfrom 1$}					
			(AK);
		
		\draw[-Stealth,\normcolor] 
			(AD) 	to[bend left=20]
					node[pos=0.4, fill=white,xshift=-2pt]{\tiny $\nm_D^K$}
			(AK);
		\draw[-Stealth,\normcolor] 
			(AL) 	to[bend left=20]
					node[fill=white]{\tiny $\nm_L^K$}
			(AK);
		\draw[-Stealth,\normcolor] 
			(AR) 	to[bend left=15]
					node[fill=white,near start]{\tiny $\nm_R^K$}
			(AK);
			
		\draw[-Stealth] 
			(AK) 	-- 
					node[rotate=-90,fill=white]{\tiny $\genUL,\genUR \mapsto 0$} 
			(AD);
		\draw[-Stealth] 
			(AK) 	-- 
					node[rotate=-45,fill=white]{\tiny $\genUL,\genUR \mapsto 0$} 
			(AR);
		\draw[-Stealth] 
			(AK) 	-- 
					node[rotate=45,fill=white]{\tiny $0 \mapsfrom \genUL,\genUR$} 
			(AL);
	\end{tikzpicture}	
\end{center}
\end{figure}

Together with the fact that 
\begin{equation}
\label{eq:resnorm}
	\res^\K_H n_e^\K(\F_2) \cong \res^\K_H n_H^\K n_e^H(\F_2) \cong \left( n_e^H(\F_2) \right)^{\boxtimes 2} \cong n_e^{H}(\F_2^{\otimes 2}) \cong n_e^{H}(\F_2)
\end{equation}
for any nontrivial proper subgroup $H$ of $\K$, this proposition allows us to determine the $\K$-Tambara functor $n_e^\K(\F_2)$, pictured in \cref{figure:NormOfZmod2}. The norms can be determined from Tambara reciprocity:  

\begin{proposition}[Tambara Reciprocity \cite{HM2019}*{Theorem 2.5}] 
\label{prop:TambaraReciprocity}
Let $H$ be any of the order $2$ subgroups of $\K$. In any $\K$-Tambara functor, we have:
\[
	\nm_H^\K(a + b) = \nm_H^\K(a) + \nm_H^\K(b) + \tr_H^\K(a\,(\gamma H \cdot b)),
\]
where $\gamma H$ is the non-identity coset of $H$ inside $\K$. 
\end{proposition}

\noindent Note that because $n_e^\K(\F_2)$ is a quotient of the Burnside $\K$-Tambara functor, the Weyl actions are trivial. Thus, in the Tambara reciprocity formula for $n_e^\K(\F_2)$, $a\,(\gamma H \cdot b)$ becomes simply $ab$. 

In particular, $\nm_L^\K(0) = 0$, $\nm_L^\K(1) = 1$, $\nm_L^\K(2) = \genUL$ and $\nm_L^\K(3) = \genUL-3$.
Similar considerations yield formulas for $\nm_D^\K$ and $\nm_R^\K$.

\begin{example}
	Recall \cite{BGHL}*{Section~5.2} that the geometric fixed points $\Phi^H \ul{M}$ of a $\K$-Mackey functor is obtained  by first quotienting at every level by all transfers up from subgroups not containing $H$ and then forgetting all levels $\K/J$, where $J$ does not contain $H$. This interacts with norms according to the formula \cite{BGHL}*{Theorem~5.15}
	\[
		\Phi^H N_J^\K \ul{M} \iso N_e^{K/H} \res^J_e \ul{M}.
	\]
	In particular, we find that
	\[
		\Phi^L N_e^\K(\F_2) \iso \Phi^L N_L^\K N_e^L(\F_2) \iso N_e^{\K/L} \res^L_e N_e^L(\F_2) \iso N_e^{\K/L}(\F_2),
	\]
	which is the Green functor
	\[
		\begin{tikzcd}[row sep=large]
			\Z/4 
				\ar[d, two heads]
				\\
			\Z/2.
				\ar[u, bend right=30, "2"', orange]
		\end{tikzcd}
	\]
	{Indeed, the elements $\genUL+\genUR-2$ and $\genUR-2$ in $N_e^\K(\F_2)$ are in the image of the transfers from $D$ and $R$, respectively. Setting those elements equal to zero produces $\Z/4$ as a quotient ring:}
	\[
		\frac{\Z/8 \oplus \Z/2\{\genUL, \genUR\}}{ -2 + \genUR, -2 + \genUL + \genUR} \cong \Z/4. 
	\]
\end{example}

\subsection{The Norm \texorpdfstring{\for{toc}{$N_{C_2}^\K \F_2$}\except{toc}{$N_{C_2}^\K(\underline{\F_2})$}}{from C_2 to K of Constant F_2}}

We can use \cref{norms as quotients} to compute the norms $n_{H}^\K(\underline{\F_2})$ for $H \in \{L,D,R\}$. We will focus on the case $H = D$; the other cases are similar. 

\begin{proposition}
	\label{top level of nDKF2}
	The value of the $\K$-Tambara functor $n_D^\K(\underline{\F_2})$ at the trivial orbit $\K/\K$ is the ring 
	\[
		n_D^\K(\underline{\F_2})(\K/\K) \cong \bigsfrac{\Z/4[\genVR]}{(2\genVR, \genVR^2),}
	\]
	where $\genVR$ is the image of $t_R+2$ under the surjection $\uA \twoheadrightarrow n_D^\K(\underline{\F_2})$. 
\end{proposition}

\begin{proof}
	Let $\uI$ be the Tambara ideal of $\uA$ generated by $2 \in \uA(\K/D)$. By \cref{norms as quotients}, $n_D^\K(\underline{\F_2})$ is the quotient of $\uA$ by $\uI$. Since $\res^D_e(2) = 2$, this ideal contains the ideal of $\uA$ generated by $2 \in \uA(\K/e)$. Hence, $n_D^\K(\underline{\F_2})$ is a further quotient of $n_e^\K(\F_2)$ by the Tambara ideal $\underline{J}$ generated by $2 \in n_e^\K(\F_2)(\K/D)$. 
	
	At the top level $n_e^\K(\F_2)(\K/\K)$, the ideal $\underline{J}(\K/\K)$ is generated by $\nm_D^\K(2)$ and $\tr_D^\K(2)$. The $\K/L$ and $\K/R$ levels do not contribute any generators because $\underline{J}(\K/L) = \underline{J}(\K/R) = 0$. By Tambara reciprocity, 
	\begin{align*}
		\nm_D^\K(2) = \nm_D^\K(1+1) &= \nm_D^\K(1) + \nm_D^\K(1) + \tr_D^\K(1) \\
		&= 1 + 1 + (\genUL + \genUR -2) \\
		&= \genUL + \genUR \\[1em]
		\tr_D^\K(2) &= 2(\genUL + \genUR - 2) \\
			&= 2 \genUL + 2\genUR - 4 \\
			&= -4
	\end{align*}
	So $n_D^\K(\underline{\F_2})(\K/\K)$ is the quotient of $n_e^\K(\F_2)(\K/\K)$ by the ideal $(\genUL + \genUR, -4)$. From \cref{prop:Norm e to K top level},
	\[
		n_e^\K(\F_2)(\K/\K) \cong \bigsfrac{\sfrac{\Z}{8}[\genUL,\genUR]}{(2\genUL,2\genUR,\genUL^2,\genUR^2,\genUL\genUR)}.
	\]
	The relation $\genUL + \genUR$ allows us to identify the two generators, and the $-4$ allows us to replace the $\Z/8$ by a $\Z/4$. Writing $\genVR$ for the image of $\genUR$ in the quotient, we have
	\[
		n_D^\K(\underline{\F_2})(\K/\K) \cong \bigsfrac{\sfrac{\Z}{4}[\genVR]}{(2\genVR,\genVR^2).} \qedhere
	\]
\end{proof}

{
\begin{remark}
	In the proof of the previous proposition, we simplified the computation of the Tambara functor  $n_D^\K(\ul{\F_2})$ by recognizing it as a quotient of $n_e^\K(\ul{\F_2})$. Such a strategy  for computing norms of $\ul{\F_2}$ should work for other groups as well. Namely, let $G$ be a finite group with a subgroup $H$ such that the $H$-Tambara functor $\ul{\F_2}$ is isomorphic to $\uA/(2_{H/H})$, where the $(2_{H/H})$ is the $H$-Tambara ideal of $\uA$ generated by $2 \in \uA(H/H)$. Then $n_e^G(\F_2) \cong \uA/(2_{G/e})$ and $n_H^G(\ul{\F_2}) \cong \uA/(2_{G/H})$. The ideal $(2_{G/H})$ contains the ideal $(2_{G/e})$ because  $\res^H_e(2) = 2$, so we may realize $n_H^G(\ul{\F_2})$ as the quotient of $n_e^G(\F_2)$ by the Tambara ideal generated by the element $2 \in n_e^G(\F_2)(G/H)$. 
	
	In particular, when $H$ is a maximal proper subgroup, $n_H^G(\ul{\F_2})(G/G)$ is the quotient of $n_e^G(\F_2)(G/G)$ by the ideal generated by $\tr_H^G(2)$ and $\nm_H^G(2)$. When $H$ is not maximal, the generating set is a little more complicated; see \cite{ideals}*{Proposition 3.4} for a description of the generators. 
\end{remark}
}

Together with the facts that $\res^\K_D n_D^\K(\underline{\F_2}) = \underline{\F_2}$ {as in \zcref{eq:resnorm}} and $\res^\K_H n_D^\K(\underline{\F_2}) \cong n_e^{C_2}(\Z/2)$ {via the double coset formula} for $H \in \{L,R\}$, \cref{top level of nDKF2} allows us to determine the $\K$-Tambara functor $n_D^\K(\underline{\F_2})$, pictured in \cref{figure:normLtoK}. 

The norm $\nm_D^\K$ is determined by $\nm_D^\K(0) = 0$ and $\nm_D^\K(1) = 1$. The norms $\nm_H^\K$ for $H \in \{L,R\}$ are determined by $\nm_H^\K(0) = 0$, $\nm_H^\K(1) = 1$, and Tambara reciprocity (\cref{prop:TambaraReciprocity}). In particular, $\nm_H^\K(2) = \genVR$ and $\nm_H^\K(3) = 1+\genVR$ for $H \in \{L,R\}$.

\begin{figure}
\caption{The $\K$-Tambara functor $n_D^\K(\ulF)$.}
\label{figure:normLtoK}
\begin{center}
	\begin{tikzpicture}[scale=0.9]
		\node (K) at (0,4) {$\bigsfrac{\sfrac{\Z}{4}[\genVR]}{(2\genVR,\genVR^2)}$};
		\node (L) at (-4,0) {$\bigsfrac{\Z}{4}$};
		\node (D) at (0,0) {$\bigsfrac{\Z}{2}$};
		\node (R) at (4,0) {$\bigsfrac{\Z}{4}$};
		\node (e) at (0,-4) {$\bigsfrac{\Z}{2}$};

		\draw[-{Stealth[] Stealth}] (L) to (e);
		\draw[-{Stealth[] Stealth}] (R) to (e);
		\draw[-Stealth] (D) to node[fill=white]{\tiny 1} (e);
		\draw[-Stealth] (K) to node[fill=white, rotate=45]{\tiny $0 \mapsfrom \genVR$} (L);
		\draw[-Stealth] (K) to node[fill=white, rotate=90]{\tiny $0 \mapsfrom \genVR$} (D);
		\draw[-Stealth] (K) to node[fill=white, rotate=-45]{\tiny $\genVR \mapsto 0$} (R);

		\draw[bend right=20,color=\inductioncolor,-Stealth] (e) to node[fill=white]{\tiny 2} (L);
		\draw[bend right=25,color=\inductioncolor,-Stealth] (e) to node[fill=white]{\tiny 2} (R);
		\draw[bend right=20,color=\inductioncolor,-Stealth] (e) to node[fill=white]{\tiny 0} (D);
		\draw[bend right=20,color=\inductioncolor,-Stealth] (L) to node[fill=white,rotate=45]{\tiny $2+\genVR$} (K);
		\draw[bend right=20,color=\inductioncolor,-Stealth] (D) to node[fill=white]{\tiny $2$} (K);
		\draw[bend right=25,color=\inductioncolor,-Stealth] (R) to node[fill=white,rotate=-45]{\tiny $2+\genVR$}(K);

		\draw[bend left=25,color=\normcolor,-Stealth] (e) to node[fill=white]{\tiny 1} (L);
		\draw[bend left=15,color=\normcolor,-Stealth] (e) to node[fill=white]{\tiny 1} (R);
		\draw[bend left=15,color=\normcolor,-Stealth] (e) to node[fill=white]{\tiny 1} (D);
		\draw[bend left=25,color=\normcolor,-Stealth] (L) to node[fill=white]{\tiny $\nm_L^\K$} (K);
		\draw[bend left=20,color=\normcolor,-Stealth] (D) to node[fill=white,inner sep=.5]{\tiny $\nm_D^\K$} (K);
		\draw[bend left=20,color=\normcolor,-Stealth] (R) to node[fill=white]{\tiny $\nm_R^\K$} (K);
	\end{tikzpicture}
\end{center}
\end{figure}

The following result about $n_D^\K \ulF$-modules will be of use later in the analysis of exact sequences of Mackey functors.

\begin{lemma} 
\label{DVanishesImplies2Torsion}
Suppose that $\ul{M}$ is a module over $n_D^\K \ulF$ such that $\ul{M}(\K/D)=0$. Then $\ul{M}(\K/\K)$ is 2-torsion.
\end{lemma}

\begin{pf}
Let $x\in \ul{M}(\K/\K)$. Since $\ul{M}(\K/D)$ vanishes, we certainly have that $\tr_D^\K \res^\K_D x$ must be zero. On the other hand, Frobenius reciprocity gives
\[
	\tr_D^\K \res^\K_D(x) = \tr_D^\K \res^\K_D(1\cdot x) = \tr_D^\K\res^\K_D(1) \cdot x = 2\cdot x.\qedhere
\]
\end{pf}


\section{Homotopical Background}
\label{sec:back}

In this section, we collect some facts that will be useful in our computations below. 

\subsection{Eilenberg--Mac Lane $\K$-Spectra}

We review here some facts about equivariant Eilenberg--Mac Lane spectra for the group $\K$. Similar statements hold more generally for any finite group.

The Eilenberg--Mac Lane spectrum of a $\K$-Mackey $\ul{M}$ functor is the $\K$-spectrum $\KH \ul{M}$ whose homotopy Mackey functors are determined by 
\[
	\upi_n \KH \ul{M} = 
	\begin{cases}
		\ul{M} & n = 0\\
		0 & n \neq 0. 
	\end{cases}
\]
Although the only nonzero integer-graded homotopy group of $\KH \ul{M}$ is $\upi_0$, suspensions of $\KH\ul{M}$ by representation spheres may have more complicated homotopy. See, for example, \cref{fig:piC2F2} in the $C_2$-equivariant case. 
An important property of the functor $\KH \colon \Mack(\K) \to \Sp^\K$ is that
it sends short exact sequences of $\K$-Mackey functors to cofiber sequences of $\K$-spectra. 

The Eilenberg--Mac Lane spectrum functor also commutes with several change of group functors. {For $H$ a subgroup of $\K$,} let $\uparrow_H^\K$ and $\downarrow^\K_H$ denote the induction and restriction functors, respectively, either between $\K$- and $H$-Mackey functors or between $\K$- and $H$-spectra. 

\begin{proposition}
	Let $\ul{M}$ be a $\K$-Mackey functor and $\ul{N}$ an $H$-Mackey functor. Then 
	\begin{enumerate}[(a)]
		\item $\downarrow_H^\K\! \KH(\ul{M}) \simeq \CH\left(\downarrow_H^\K\!\ul{M}\right)$
		\item $\uparrow_H^\K\! \CH(\ul{N}) \simeq \KH\left(\uparrow_H^\K\!\ul{N}\right)$.
	\end{enumerate}
\end{proposition}

{
We include a proof of this well-known result as we were not able to find a reference in the literature.
}

\begin{pf}
{
Consider the diagram of adjunctions
}
\[
\begin{tikzcd}
\Sp^{\K}_{\geq 0}
\ar[d,"\downarrow^\K_H" swap, xshift={-.5ex}]
\ar[r,"\upi_0",yshift={.5ex}]
& 
\Mack(\K)
\ar[l,"\KH",yshift={-.5ex}]
\ar[d,"\downarrow^\K_H" swap, xshift={-.5ex}]
\\
\Sp^H_{\geq 0}
\ar[u,"\uparrow^\K_H" swap, xshift={.5ex}]
\ar[r,"\upi_0",yshift={.5ex}]
& 
\Mack(H).
\ar[l,"\CH",yshift={-.5ex}]
\ar[u,"\uparrow^\K_H" swap, xshift={.5ex}]
\end{tikzcd}
\]
{
We claim that the diagram of left adjoints commutes. Namely, we claim that for every (connective) $\K$-spectrum $X$, we have a natural isomorphism of $H$-Mackey functors $\upi_0 \! \downarrow^\K_H \!  X \iso \, \downarrow^\K_H \! \upi_0 X$. Indeed, we have isomorphisms, natural in $H$-sets $T$,
\[
\begin{split}
\upi_0 (\downarrow^\K_H \!  X) (T) \iso [ T_+, \downarrow^\K_H \! X]^H
\iso [ (\K\times_H T)_+, X]^\K \iso \, \downarrow^\K_H \! (\upi_0 X) (T).
\end{split}
\]
It follows that the diagram of right adjoints commutes, which is item (b).

On the other hand, both vertical adjunctions are ambidextrous. Thus the isomorphism (a) would follow from an isomorphism of $\K$-Mackey functors $\upi_0  \! \uparrow^\K_H \! Y \cong \, \uparrow^\K_H \! \upi_0 Y$, natural in $H$-spectra $Y$. This is provided by the isomorphism, natural in $\K$-sets $W$,
\[
\upi_0 ( \uparrow^\K_H \! Y) (W) = [W_+, \uparrow_H^\K Y]^\K
\cong [ \downarrow^\K_H W_+, Y]^H \cong \ \uparrow_H^\K \!(\upi_0 Y) (W).
\]
Since the diagram of left adjoints commutes, so does the diagram of right adjoints. This establishes (a).
}
\end{pf}

However, taking the Eilenberg--Mac Lane spectrum does not commute with norms. 
{For example, $N_e^{C_2}H\F_2$ is not Eilenberg--Mac~Lane, as its underlying spectrum is $H\F_2\wedge H\F_2$, whose homotopy groups are the dual Steenrod algebra. On the other hand, for $X$ a connective $H$-spectrum, there is an isomorphism $\upi_0 N_H^\K X \cong N_H^\K \upi_0 X$ \cite{U}*{Corollary~3.2}.
}

\subsection{\texorpdfstring{$\RO(\K)$}{RO(K)}-Graded Suspensions}

The abelian group $\RO(\K)$ is a free abelian group of rank 4 generated by the trivial representation $1$ and three one-dimensional representations $\sigma_L$, $\sigma_D$, and $\sigma_R$. For a subgroup $H \in \{L,D,R\}$ of $\K$, the representation $\sigma_H$ corresponds to the sign representation $\sigma$ of $\K/H\iso C_2$, where $\K$ acts 
via the quotient homomorphism $\K \to \K/H$. We also write $\rho = 1 + \sigma_L + \sigma_D + \sigma_R$ for the regular representation of $\K$ and $\orho = \rho - 1$ for the reduced regular representation.

Because we are interested in the $\RO(\K)$-graded homotopy of $\K$-spectra, we will frequently suspend by (virtual) representation spheres $S^V$ for $V \in \RO(\K)$. We are primarily interested in $\orho$-suspensions of $\K$-spectra, but we will also find occasion to use $\sigma_H$-suspensions.  
The suspension $\Sigma^{\sigma_H}$ fits into a cofiber sequence 
\[
	\K/H_+ \wedge X \to X \to \Sigma^{\sigma_H} X.
\]
The fiber has an alternative description, by the shearing isomorphism:

\begin{proposition}[{Shearing}] Let $X$ be a $\K$-spectrum and let 
$H$ be a subgroup of $\K$.
Then 
\[	
	\K/H_+ \wedge X \simeq\, \uparrow^\K_H\downarrow^\K_H\! X,
\]
where $\downarrow^\K_H$ denotes the restriction from $\K$-spectra to $H$-spectra, and $\uparrow_H^\K$ the induction.
\end{proposition}

In particular, when $X$ is an Eilenberg--Mac Lane spectrum, we have 
\[
	\K/H_+ \wedge \KH \ul{M} \simeq
	\uparrow^\K_H\downarrow^\K_H \KH \ul{M}
	\simeq
	 \KH\left(\uparrow^\K_H\downarrow^\K_H\!\ul{M}\right).
\]
A useful consequence of this is the following:

\begin{corollary}
\label{cor:sigmaH suspensions of trivial H}
	Let $X$ be a $\K$-spectrum and let $H \in \{L,D,R\}$. If the restriction of $X$ to $H$ is contractible, then 
	\[
		\Sigma^{\sigma_H} X \simeq X.
	\]
\end{corollary}

\subsection{Inflation Functors}
Our main computations are $\K$-equivariant. However, we often make comparisons to $C_2$-equivariant computations, as $C_2$ appears both as a subgroup and a quotient of $\K$. For reference, we display in \cref{fig:piC2F2} and \cref{fig:piC2NF2} the $RO(C_2)$-graded homotopy Mackey functors of $\CH \ulF$ and $\CH N_e^{C_2}\F_2$, respectively. The Mackey functors appearing in those charts are as shown in \cref{tab-C2Mackey}.

We follow \cite{Slice}*{Section~4} in writing $\phi^*_H \ul{M}$ for the inflation along the quotient $\K \to \K/H$ 
of the $\K/H$-Mackey functor $\ul{M}$, for any subgroup $H \leq \K$. 
Since the groups $\K/H$ are canonically isomorphic for all $H$ in $\{L,D,R\}$, we write $\phi^*_{LDR}\ul{M}$ as shorthand for the sum $\phi^*_L \ul{M} \oplus \phi^*_D \ul{M} \oplus \phi^*_R \ul{M}$ for any $C_2$-Mackey functor $\ul{M}$. See \cref{figure:inflations}.
As in \cite{GY}, we write $\ulg$ for the fully inflated Mackey functor $\phi^*_{\K}\F$ (see \cref{tab-K4Mackey}). 

\begin{table}
	\caption{The inflation functor $\phi_H^* \colon \Mack(C_2) \to \Mack(\K)$, where we identify $C_2 \cong \K/H$. We also write $\phi^*_{LDR} \ul M := \phi_L^* \ul M \oplus \phi_D^* \ul M \oplus \phi_R^* \ul M$.}
	\label{figure:inflations}
	\begin{center}
		\begin{tabular}{|c|c|c|c|}
		\hline
		$\ul M$ & $\phi_L^* \ul M$ & $\phi_D^* \ul M$ & $\phi_R^* \ul M$ 
		\\
		\hline
		\begin{tikzpicture}
			\node (Mtop) at (0,1.5) {$\ul M(C_2/C_2)$};
			\node (Mbot) at (0,0) {$\ul M(C_2/e)$};
			\node at (0,-0.75) {\phantom{x}};
			\draw[->] (Mtop) --node[left]{\footnotesize $\res$} (Mbot); 
			\draw[->,\inductioncolor] (Mbot) to[bend right=30] node[right]{\footnotesize $\tr$} (Mtop);
		\end{tikzpicture}
		&
		\begin{tikzpicture}
			\node (K) at (0,1.5) {$\ul M(C_2/C_2)$};
			\node (L) at (-1.5,0) {$\ul M(C_2/e)$};
			\node (D) at (0,0) {$0$};
			\node (R) at (1.5,0) {$0$};
			\node (e) at (0,-1.5) {$0$};
			\draw[->] (K) --node[left]{\footnotesize $\res$} (L); 
			\draw[->,\inductioncolor] (L) to[bend right=30] node[right]{\footnotesize $\tr$} (K);
		\end{tikzpicture}
		&
		\begin{tikzpicture}
			\node (K) at (0,1.5) {$\ul M(C_2/C_2)$};
			\node (L) at (-1.5,0) {$0$};
			\node (D) at (0,0) {$\ul M(C_2/e)$};
			\node (R) at (1.5,0) {$0$};
			\node (e) at (0,-1.5) {$0$};
			\draw[->] (K) --node[left]{\footnotesize $\res$} (D); 
			\draw[->,\inductioncolor] (D) to[bend right=30] node[right]{\footnotesize $\tr$} (K);
		\end{tikzpicture}
		&
		\begin{tikzpicture}
			\node (K) at (0,1.5) {$\ul M(C_2/C_2)$};
			\node (L) at (-1.5,0) {$0$};
			\node (D) at (0,0) {$0$};
			\node (R) at (1.5,0) {$\ul M(C_2/e)$};;
			\node (e) at (0,-1.5) {$0$};
			\draw[->] (K) --node[left]{\footnotesize $\res$} (R); 
			\draw[->,\inductioncolor] (R) to[bend right=30] node[right]{\footnotesize $\tr$} (K);
		\end{tikzpicture}
		\\
		\hline
		\end{tabular}
	\end{center}
\end{table}

The inflation functor $\phi_H^*$ in fact extends to {a ``geometric inflation'' functor on} spectra, and $\phi_H^* \CH \ul{M} \simeq \KH \phi_H^* \ul{M}$. 
In fact, by considering the Postnikov tower for any $\K/H$-spectrum $X$, one gets more generally an isomorphism of $\K$-Mackey functors
\begin{equation}
\label{eq:inflatinghomotopy}
	\upi_n \phi_H^* X \iso  \phi_H^* \upi_n X.
\end{equation}
We will use this frequently.
More importantly, for any $\K$-representation $V$, we have \cite{Slice}*{Corollary~4.6}
\begin{equation}
\label{eq:inflatingsuspensions}
\Sigma^V \KH \phi_H^* \ul{M} \simeq \phi_H^* \left( \Sigma^{V^H} {\CH} \ul{M} \right)
\end{equation}
where $V^H$ is considered as a representation of $\K/H$.

This, combined with \cref{fig:piC2F2}, gives the following useful equivalences:

\begin{prop}
\label{BackgroundSrhoHM}
We have equivalences
\begin{enumerate}[(a)]
\item $\Sigma^{\orho} \KH \ulg \simeq \KH \ulg$
\item $\Sigma^{\orho} \KH \phi_H^* \ulF^* \simeq \Sigma^1 \KH \phi_H^*\ulf$ for $H\in \{L,D,R\}$
\item $\Sigma^{\orho} \KH \phi_H^* \ulf \simeq \Sigma^1 \KH \phi_H^*\ulF$ for $H\in \{L,D,R\}$.
\end{enumerate}
where $\ulf$ is as displayed in \cref{tab-C2Mackey}. 
\end{prop}

\section{The Homotopy of \texorpdfstring{$\KH N_e^\K \F_2$}{H N_e^K F_2}}
\label{sec:HNeF}

In this section, we compute the homotopy Mackey functors {$\upi_{\blacklozenge} \KHN$, where $\KN$ will often be used as an abbreviation for $ N_e^\K \bbF_2$.} The results are displayed in \cref{fig:HN}.

\begin{remark}
For $k\geq 0$, the homology groups $\ul{H}_n(S^{k\orho};\KN) \iso \upi_{n-k\orho}\KHN$ are concentrated in degrees $n\geq 0$. We will refer to this portion of $\upi_{\blacklozenge} \KHN$ as the \textbf{positive cone}. It appears in the fourth quadrant of \cref{fig:HN}. Similarly, the cohomology groups $\ul{H}^n(S^{k\orho};\KN) \iso \pi_{-n+k\orho} \KHN$ are concentrated in degrees $n\geq 0$. We will refer to this portion as the \textbf{negative cone}. It appears in the second quadrant of \cref{fig:HN}.
\end{remark}

We proceed with a computation of the positive cone for $\KHN$ followed by that of the negative cone.

\subsection{The Positive Cone of \texorpdfstring{$\KH N_e^{\K} \bbF_2$}{$HN$}}
\label{sec:PinkposorhoHN}

Here we compute the homotopy Mackey functors of the $\K$-spectra $\Sigma^{k\orho} \KHN$ for $k\geq 0$.

We start with the case $k=1$.
Our analysis will rely on cofiber sequences of equivariant Eilenberg--Mac~Lane spectra arising from short exact sequences of Mackey functors. In particular, the following Mackey functor $\ul{E}$ will be of use.

\noindent
\begin{minipage}{0.5\textwidth}
\begin{defn} 
Let $\ul{E}$ be the cokernel of $\ulF^* \into N_e^{\K} \bbF_2$. This Mackey functor is displayed to the right. 
\end{defn}
\end{minipage}
\quad
\begin{tikzcd}[row sep={7ex}, column sep={4ex},ampersand replacement=\&]
 \& \Z/4\oplus \Z/2\{{\genUL},{\genUR}\} \ar[dl,"\mymidsizematrix{1 & 0 &  0}" swap,pos=0.4,shift right, bend right=5] \ar[d,"\mymidsizematrix{1 & 0 &  0}" swap,pos=0.5,shift right=0.5] \ar[dr,"\mymidsizematrix{1 & 0 &  0}" swap,pos=0.8,shift right=0.5, bend left=5]\\ 
 \Z/2 \ar[ur,"\mymidsizematrix{2 \\ 1 \\ 0}" swap, pos=0.3, bend left=5,shift right=0.5] \& \Z/2 \ar[u,"\mymidsizematrix{2 \\ 1 \\ 1}" swap, pos=0.3,shift right=0.5]  \& \Z/2 \ar[ul,"\mymidsizematrix{2 \\ 0 \\ 1}" swap, pos=0.3,bend right=5,shift right] \\ 
 \& 0
\end{tikzcd}
\\

The Mackey functor $\ul{E}$ also fits into a short exact sequence 
\begin{equation}
\label{ses:E}
	\phi^*_{LDR} (\ulF^*) \into \ul{E} \onto \ulg.
\end{equation}

\begin{lemma}
\label{lemma:homotopy of E}
The nonzero homotopy Mackey functors of $\Sigma^{\orho}\KH\ul{E}$ are 
\[
	\ul\pi_n \Sigma^{\orho} \KH \ul{E} = 
	\begin{cases}
		\ul{g} & n = 0\\
		\phi^*_{LDR}\ul{f} & n = 1
	\end{cases}
\]
\end{lemma}

\begin{proof}
	Applying the functors $\Sigma^{\orho}$ and $\KH$ to the short exact sequence \eqref{ses:E} yields a cofiber sequence:
	\[
	\Sigma^{\orho} \KH \phi_{LDR}^* \ulF^* \to \Sigma^{\orho} \KH \ul{E} \to \Sigma^{\orho} \KH \ulg.
\]
But {\cref{BackgroundSrhoHM} provides equivalences } $\Sigma^{\orho} \KH \phi_{LDR}^* \ulF^* \simeq\Sigma^1 \KH \phi_{LDR}^* \ul{f}$
and $\Sigma^{\orho} \KH \ulg \simeq \KH \ulg$,
so the homotopy of $\Sigma^{\orho} \KH \ul{E}$ can be read off from the associated long exact sequence.
\end{proof}

\begin{prop}
\label{prop:PiSiRhoKHN}
The nonzero homotopy Mackey functors of $\Sigma^{\orho}  \KH N_e^{\K} \bbF_2$ are
\[
\ul\pi_n \Sigma^{\orho}  \KH N_e^{\K} \bbF_2 \iso
\begin{cases}
\ul{\bbF}_2 & n=3 \\
\phi_{LDR}^* \ul{f} & n=1 \\
\ul{g} & n=0.
\end{cases}
\]
\end{prop}

\begin{pf}
By {applying the functors $\Sigma^{\orho}$ and $\KH$ to the defining short exact sequence for} $\ul{E}$, we have a cofiber sequence 
\[
	\Sigma^{\orho}\KH \ulF^* \to \Sigma^{\orho}\KHN \to \Sigma^{\orho}\KH \ul{E}.
\]
{We understand the homotopy of the fiber:} $\Sigma^{\orho} \KH \ulF^*$ is $\Sigma^3 \KH \ulF$ by \cite{GY}*{Proposition~4.2}. {We also understand the homotopy of the cofiber by \cref{lemma:homotopy of E}. The desired homotopy may then be read off from the associated long exact sequence.}
\end{pf}

The computation of the homotopy of $\Sigma^{\orho} \KHN$ gives a fiber sequence
\begin{equation}
\label{eq:Separate_korhoKHN}
	\Sigma^3 \KH \ulF \to \Sigma^{\orho} \KHN \to \Sigma^{\orho} \KH \ul{E}.
\end{equation}
It turns out that the $(k-1)\orho$-suspension of the map $\Sigma^3 \KH \ulF \to \Sigma^{\orho} \KHN $ detects much of the homotopy of $\Sigma^{k\orho} \KHN $, as follows from analysis of the suspensions of the cofiber.

\begin{prop}
\label{prop:PInP01}
For $k\geq 2$, 
the nonzero homotopy Mackey functors of $\Sigma^{k\orho}  \KH \ul{E}$ are
\[
\ul\pi_n \Sigma^{k\orho}  \KH \ul{E} \iso
\begin{cases}
\phi_{LDR}^* \ulF & n=k \\
\ul{g}^3 & n \in [2,k-1] \\
\ul{g} & n=0.
\end{cases}
\]
\end{prop}

\begin{pf}
As in \cref{lemma:homotopy of E}, we have a fiber sequence
\[
	\Sigma^1 \KH \phi_{LDR}^*\ul{f} \to 
	\Sigma^{\orho} \KH \ul{E}
	\to \KH \ulg.
\]
Suspending this $(k-1)\orho$ times yields another fiber sequence 
\[
	\Sigma^{1 + (k-1)\orho}\KH \phi_{LDR}^*\ul{f} \to 
	\Sigma^{k\orho} \KH \ul{E}
	\to \Sigma^{(k-1)\orho}\KH \ulg.
\]
Applying the equivalences \[
	\Sigma^{\orho} \KH \ulg \simeq \KH \ulg
	\qquad \text{and} \qquad
	\Sigma^{\orho} \KH \phi_{LDR}^* \ul{f} \simeq \Sigma^1 \KH \phi_{LDR}^* \ulF
\]
from \cref{BackgroundSrhoHM}, we see that this fiber sequence is equivalent to 
\begin{equation}
\label{cofiber sequence for krho suspension}
	\Sigma^{2 + (k-2)\orho}\KH \phi_{LDR}^*\ulF \to 
	\Sigma^{k\orho} \KH \ul{E}
	\to \KH \ulg.
\end{equation}
Then by \zcref{eq:inflatingsuspensions}, the left term becomes 
\[
	\Sigma^{2 + (k-2)\orho}\KH \phi_{LDR}^*\ulF 
	\simeq 
	\bigvee_{H \in \{L,D,R\}} 
	\phi_{H}^* \left(\Sigma^{(2 + (k-2){\sigma})} \CH \ulF \right)
\]
{Here we have used the identification $\orho^H \iso \orho=\sigma$ for $H\in \{L,D,R\}$.}
We may compute the homotopy of the above using \zcref{eq:inflatinghomotopy} and the computation of the $C_2$-Mackey functors $\ul{\pi}_\blacklozenge^{C_2} \CH \ulF$, as in \cref{fig:piC2F2}. The result then follows from the long exact sequence associated to \zcref{cofiber sequence for krho suspension}.
\end{pf}

In the case $k=2$, the cofiber sequence \zcref{eq:Separate_korhoKHN} and \cref{prop:PInP01} give the following computation of $\Sigma^{2\orho} \KHN$.

\begin{cor}
The nontrivial homotopy Mackey functors of $\Sigma^{2\orho} \KH N_e^\K \F_2$ are
\[
\ul{\pi}_n \Sigma^{2\orho} \KH N_e^K \bbF_2 \iso
\begin{cases}
\ul{\pi}_{n-3} \Sigma^{\orho} \KH \ul{\bbF}_2 & n \geq 3 \\
\phi_{LDR}^* \ul{\bbF}_2  & n=2 \\
\ul{g} & n=0.
\end{cases}
\]
\end{cor}

More generally, combining \eqref{eq:Separate_korhoKHN} with 
\cref{prop:PInP01} 
gives the following.

\begin{cor}
\label{cor:PI>=K+1HKN}
For $k \geq 2$, we have isomorphisms 
\[
	\ul{\pi}_n \Sigma^{k\orho} \KH N_e^\K \F_2 \iso \ul{\pi}_n \Sigma^{3+(k-1)\orho} \KH \ulF = \ul{\pi}_{n-3} \Sigma^{(k-1)\orho} \KH \ulF
\]
for $n \geq k+1$.
\end{cor}

\medskip

However, starting with $k=3$, the homotopy of the left and right terms in the sequence
\begin{equation}
\label{eq:SIkorhoKHN}
	\Sigma^{3+(k-1)\orho} \KH \ulF \to \Sigma^{k\orho} \KHN \to \Sigma^{k\orho} \KH \ul{E}
\end{equation}
begin to overlap.
In order to find the lower homotopy Mackey functors of $\Sigma^{k\orho} \KHN$, it is convenient to consider a different sequence.
The kernel of the surjection $\KN \twoheadrightarrow \ul{\bbF}_2$ is $\ul{B}(2,0) \oplus \ul{g}^2$, and 
we now describe its homotopy.

\begin{prop}
\label{prop:HtpySIkorhoB20}
The nonzero homotopy Mackey functors of $\Sigma^{\orho} \KH (\ul{B}(2,0) \oplus \ul{g}^2)$  are
\[
\ul{\pi}_n \Sigma^{\orho} \KH(\ul{B}(2,0) \oplus \ul{g}^2) \iso
\begin{cases}
\ul{mg} & n=1 \\
\ul{g}^3 & n=0.
\end{cases}
\]
For $k \geq 2$, 
the nonzero homotopy Mackey functors of $\Sigma^{k\orho} \KH (\ul{B}(2,0) \oplus \ul{g}^2)$ are
\[
\ul{\pi}_n \Sigma^{k\orho} \KH (\ul{B}(2,0) \oplus \ul{g}^2) \iso
\begin{cases}
\phi_{LDR}^* \ul{\bbF}_2 & n=k \\
\ul{g}^3 & n \in [2,k-1] \\
\ul{g}^2 & n=1 \\
\ul{g}^3 & n=0.
\end{cases}
\]
\end{prop}

\begin{pf}
Since $\Sigma^{k\orho} \KH \ulg$ is equivalent to $\KH \ulg$, the claim amounts to the computation of $\Sigma^{k\orho} \KH \ul{B}(2,0)$.

The Mackey functor $\ul{B}(2,0)$ is the cokernel of the inclusion $\ulZ^* \into \ulZ$, so the homotopy of $\Sigma^{\orho} \KH \ul{B}(2,0)$ can be calculated from the cofiber sequence
\[
\begin{tikzcd}[row sep=3mm]
	\Sigma^{\orho} \KH \ulZ^* 
		\ar[r]
		& 
	\Sigma^{\orho}\KH \ulZ 
		\ar[r]
		& 
	\Sigma^{\orho} \KH \ul{B}(2,0)
\end{tikzcd}
\]
and the equivalence $\Sigma^{\orho} \KH \ulZ^* \simeq \Sigma^3 \KH \ulZ$ \cite{Slone}*{Proposition~4.2}, together with the computation of $\upi_n \Sigma^{\orho} \KH \ulZ$ \cite{Slone}*{Proposition~9.1}. 
Note that \cite{Slone} reports the homotopy of the $\rho = (1 + \orho)$-suspensions of $\KH\ulZ$, whereas we are interested in the $\orho$-suspensions. The long exact sequence associated to the cofiber sequence above immediately yields $\upi_0 \Sigma^{\orho} \KH \ul{B}(2,0)$ and $\upi_1 \Sigma^{\orho} \KH \ul{B}(2,0)$. It also shows that $\upi_n \Sigma^{\orho} \KH \ul{B}(2,0)$ vanishes in degrees $n < 0$, $n = 2$, and $n > 4$. In degrees 3 and 4, the long exact sequence is
\[
\begin{tikzcd}[arrows={thick},]
		&
	\Sigma^3 \KH \ulZ 
		&
	\Sigma^{\orho} \KH \ulZ
		& 
	\Sigma^{\orho} \KH \ul{B}(2,0)
		\\
	n = 4
		&
	\phantom{0}
		&
	0 
		\ar[r] 
		\ar[d, phantom, ""{coordinate, name=A}]
		&
	\upi_4 \Sigma^{\orho} \KH \ul{B}(2,0)
		\ar[dll, rounded corners, 
			to path = { -- ([xshift=5mm]\tikztostart.east)
						|- (A)
						-| ([xshift=-5mm]\tikztotarget.west)
						-- (\tikztotarget)
						}] 
		\\
	n = 3
		&
	\ulZ 
		\ar[r] 
		& 
	\ulZ 
		\ar[r]
		\ar[d, phantom, ""{coordinate, name=B}]
		&
	\upi_3 \Sigma^{\orho} \KH \ul{B}(2,0)
		\ar[dll, rounded corners,
			to path = { -- ([xshift=5mm]\tikztostart.east)
						|- (B)
						-| ([xshift=-5mm]\tikztotarget.west)
						-- (\tikztotarget)
						}]  
		& 
		\\
	n = 2
		&
	0	
		&
	\phantom{0}
\end{tikzcd}
\] 
Since $\ul{B}(2,0)$ vanishes at the underlying level, the underlying level of the map $\ulZ \to \ulZ$ must be an isomorphism. Because these are constant Mackey functors, this determines the homomorphism $\ulZ \to \ulZ$ entirely; it must be an isomorphism. Hence, $\upi_3 \Sigma^{\orho} \KH \ul{B}(2,0)$ and $\upi_4 \Sigma^{\orho} \KH \ul{B}(2,0)$ are zero.

We next turn to the calculation of the $k\orho$-suspensions of $\KH \ul{B}(2,0)$. Given the calculation of $\ul{\pi}_n \Sigma^{\orho} \KH \ul{B}(2,0)$ in the previous paragraph, there is a Postnikov sequence
\[
	\Sigma^1 \KH \ul{mg} \to \Sigma^{\orho} \KH \ul{B}(2,0) \to \KH \ulg.
\]
To use this sequence to compute the homotopy of $\Sigma^{k\orho}\ul{B}(2,0)$, we must first compute the homotopy of $\Sigma^{1 + (k-1)\orho} \KH \ul{mg}$.
As in \cite{GY}*{Proposition~7.4}, there is a short exact sequence of Mackey functors $\phi_{LDR}^* \ul{f} \into \ul{mg} \to \ulg^2$, which gives a fiber sequence
\[
	\Sigma^1 \KH \phi_{LDR}^* \ulF = \Sigma^{\orho} \KH \phi_{LDR}^* \ul{f} \to \Sigma^{\orho} \KH \ul{mg} \to \Sigma^{\orho} \KH \ulg^2 = \KH \ulg^2.
\]
Suspending again by $\Sigma^{(k-2)\orho}$ gives a fiber sequence
\[
	\Sigma^{1 + (k-2)\orho} \KH \phi_{LDR}^* \ulF
	\to 
	\Sigma^{(k-1) \orho} \KH \ul{mg} 
	\to 
	\KH \ul{g}^2.
\]
Using \zcref{eq:inflatingsuspensions}, we may calculate the homotopy of $\Sigma^{1+(k-2)\orho} \KH \phi_{LDR}^* \ulF$ as in the proof of \cref{prop:PInP01} (see also \cref{fig:piC2F2}). Unwinding the cofiber sequences yields the homotopy of $\Sigma^{\orho} \KH B(2,0)$.
\end{pf}

We now use the previous results to determine the Mackey functors $\upi_n \Sigma^{k\orho} \KHN$ for $k \geq 3$.

\begin{theorem} 
\label{thm:PInSigmakrhoHKN}
For $k \geq 3$,
the nonzero homotopy Mackey functors of $\Sigma^{k\orho} \KH N_e^\K \F_2$ are
\[
\ul{\pi}_n \Sigma^{k\orho} \KH N_e^\K \F_2 \iso
\begin{cases}
\ul{\pi}_n \Sigma^{k\orho} \KH \ul{\bbF}_2 & n \geq k+2 \\
\ul{g}^{2k-3} & n=k+1 \\
\phi_{LDR}^* \ul{\bbF}_2 \oplus \ul{g}^{2k-5} & n=k \\
\ul{g}^{2n-2} & n\in[3,k-1] \\
\ul{g}^3 & n=2 \\
\ul{g} & n=0.
\end{cases}
\]
\end{theorem}

\begin{pf}
The homotopy Mackey functors in dimensions 0, 1, and 2 are given by \cref{prop:PInP01}.
Those in dimensions at least $k+1$ are given by \cref{cor:PI>=K+1HKN}.
The answer is stated differently here to emphasize the relation to $\Sigma^{k\orho} \KH \ulF$. These spectra are related via the cofiber sequence
\begin{equation}
\label{cofib:augHNHF}
	\Sigma^{k\orho} \KH \ul{B}(2,0) \oplus \ulg^2 \to \Sigma^{k\orho} \KHN \to \Sigma^{k\orho} \KH \ulF.
\end{equation}
By \cref{prop:HtpySIkorhoB20}, the map $\Sigma^{k\orho} \KHN \to \Sigma^{k\orho} \KH \ulF$ induces an isomorphism of homotopy Mackey functors in degrees at least $k+2$. In fact, we claim that it is an {\it injection} in all degrees.
Consider, for instance the long exact sequence associated to \zcref{cofib:augHNHF} for $k=3$, where the only Mackey functor remaining to be determined is $\upi_3 \Sigma^{3\orho} \KHN$.
\[
\begin{tikzcd}[arrows={thick,}]
		& 
	\upi_n \Sigma^{3\orho} \KH \ul{B}(2,0) \oplus \ulg^2 
 		& 
	\upi_n \Sigma^{3\orho} \KHN 
		& 
	\upi_n \Sigma^{3\orho} \KH \ulF 
		\\
 	n=4 
		\ar[r, phantom, ""{coordinate, name=X}]
		& 
	0  \ar[r, "0"]
		\ar[from=X, to=2-2, dashed]
		& 
	\ulg^3 \ar[r,hookrightarrow]  
		\ar[d, phantom, ""{coordinate, name=A}]
		& 
	\phi_{LDR}^* \ulF \oplus \ulg^3
		\ar[dll, rounded corners, 
			to path={ -- ([xshift=5mm]\tikztostart.east)
					  |- (A) 
					  -| ([xshift=-5mm]\tikztotarget.west)
					  -- (\tikztotarget) 
					}]
		\\
 	n=3 
		& 
	\phi_{LDR}^* \ulF \ar[r,"0"] 
		& 
	? \ar[r,hookrightarrow]
		\ar[d, phantom, ""{coordinate, name=B}] 
		& 
	\phi_{LDR}^* \ulF \oplus \ulg^4
		\ar[dll, rounded corners, 
			to path={ -- ([xshift=5mm]\tikztostart.east)
					  |- (B) 
					  -| ([xshift=-9mm]\tikztotarget.west)
					  -- (\tikztotarget) 
					}]
 		\\
 	n=2 
		& 
	\ulg^3 \ar[r,"0"]  
		& 
	\ulg^3 \ar[r,hookrightarrow] 
		\ar[d, phantom, ""{coordinate, name=C}]
		& 
	\ulg^5
		\ar[dll, rounded corners, 
			to path={ -- ([xshift=13mm]\tikztostart.east)
					  |- (C) 
					  -| ([xshift=-9mm]\tikztotarget.west)
					  -- (\tikztotarget) 
					}]
 		\\
	n=1 
		& 
	\ulg^2 
		\ar[r, "0"]
		& 
	0 	
		\ar[r, "0"]
		& 
	\ulg^3
		\ar[r, phantom, ""{coordinate, name=Y}, near end]
		\ar[from=5-4,to=Y, dashed]
		&
	\phantom{a}
\end{tikzcd}
\]
The connecting homomorphism $\upi_4 \Sigma^{3\orho} \KH \ulF \to \upi_3 \KH \ul{B}(2,0) \oplus \ulg^2$ is an isomorphism upon restricting to any $C_2$, which forces it to be a surjection of Mackey functors. Thus $\upi_3 \Sigma^{3\orho} \KHN$ is a sub-Mackey functor of $\upi_3 \Sigma^{3\orho} \KH \ulF \iso \phi_{LDR}^* \ulF \oplus \ulg^4$.
On the other hand, the sequence \eqref{eq:SIkorhoKHN} forces a short exact sequence
\[ 
	\ulg \into  \upi_3 \Sigma^{3\orho}\KHN \onto \phi_{LDR}^* \ulF
\]
This combines to force an isomorphism $ \upi_3 \Sigma^{3\orho} \KHN \iso \phi_{LDR}^* \ulF \oplus \ulg$.
A similar argument works for higher values of $k$.
\end{pf}

\begin{remark}
\label{rem:HNHFComp}
We argued that the map $\Sigma^{k\orho}\KHN \to \Sigma^{k\orho} \KH \ulF$ induces an {isomorphism} on Mackey functors in degrees at least $k+2$ by showing that the fiber in \zcref{cofib:augHNHF} has no homotopy above degree $k$, in \cref{prop:HtpySIkorhoB20}. 
In other words, we used that the homology of $S^{k\orho}$ with coefficients in $\ul{B}(2,0) \oplus \ulg^2$ vanishes above degree $k$.
An alternative argument for this is that $S^{k\orho}$ has a $\K$-CW structure in which all cells in degrees $k+1$ or higher are $\K$-free. 
In particular, it is the $\K$-CW structure associated to the $k$-fold smash product $\left(S^{\orho}\right)^{\wedge k} \simeq S^{k\orho}$. We need only observe that $\K/H \times \K/J \cong \K/e$ for any two distinct order two subgroups $H$ and $J$. Since $ \ul{B}(2,0) \oplus \ulg^2$ vanishes at the underlying level, the homology of $S^{k\overline{\rho}}$ with these {coefficients} will vanish in degrees at least $k+1$. 
\end{remark}

The blue shading in \cref{fig:HF,fig:HN}  highlights the regions
{in which the isomorphism $\upi_n \Sigma^{k\orho} \KHN \iso \upi_n \Sigma^{k\orho} \KH \ulF$ was shown to hold in the positive cones.}

\subsection{The Negative Cone of \texorpdfstring{$\KH N_e^{\K} \bbF_2$}{$HN$}}
\label{sec:PinknegorhoHN}

We now turn to the case of $\Sigma^{k\orho} \KH N_e^{\K} \bbF_2$ for $k$ negative. The strategy is largely the same as in \cref{sec:PinkposorhoHN}: most of the answer follows easily from a cofiber sequence, while some extension problems are resolved by considering a separate cofiber sequence.

We will again use the Mackey functor $\ul{B}(2,0)$. An argument as in the proof of \cref{prop:HtpySIkorhoB20} gives the following computation.

\begin{prop}
\label{prop:HtpySInegkorhoB20}
The nonzero homotopy Mackey functors of $\Sigma^{-\orho} \KH \ul{B}(2,0) \oplus \ul{g}^2$  are
\[
\ul{\pi}_n \Sigma^{-\orho} \KH \ul{B}(2,0) \oplus \ul{g}^2 \iso
\begin{cases}
\ul{g}^3 & n=0 \\
\ul{mg}^* & n=-1.
\end{cases}
\]
For $k \geq 2$, 
the nonzero homotopy Mackey functors of $\Sigma^{-k\orho} \KH \ul{B}(2,0) \oplus \ul{g}^2$ are
\[
\ul{\pi}_n \Sigma^{k\orho} \KH \ul{B}(2,0) \oplus \ul{g}^2 \iso
\begin{cases}
\ul{g}^3 & n=0 \\
\ul{g}^2 & n=-1 \\
\ul{g}^3 & n \in [-k+1,-2] \\
\phi_{LDR}^* \ul{\bbF}_2^* & n=-k.
\end{cases}
\]
\end{prop}

\begin{remark}
Rather than arguing as in \cref{prop:HtpySIkorhoB20}, an alternative method to obtain \cref{prop:HtpySInegkorhoB20} is to use Brown-Comenetz duality, as in \cite{Slone}*{Section~8}, since the Mackey functors $\ul{B}(2,0)$ and $\ul{g}$ are self-dual.
\end{remark}

Returning to $\Sigma^{-k\orho} \KHN$, the lower homotopy groups are captured by suspensions of $\KH \ulF$, as we now describe (see also \cref{fig:FiberSeqnSiHN}). {In \cref{fig:HF,fig:HN}, we use red shading to indicate the region of the negative cone where $\pi_{n+k\orho} \KH \ulF$ and $\pi_{n + k \orho} \KHN$ agree.} 

\begin{figure}
\caption{The fiber sequence $\upi_* \Sigma^{-4\orho} \KH \ul{B}(2,0) \oplus \ulg^2 \to \upi_* \Sigma^{-4\orho} \KHN \to \upi_* \Sigma^{-4\orho} \KH \ulF$. In the green region, the homotopy of $\Sigma^{-4\orho}\KHN$ matches the homotopy of the fiber, and in the red region, the homotopy is the same as the homotopy of $\KH\ulF$.}
\label{fig:FiberSeqnSiHN}
\begin{tikzpicture}
\filldraw[color=red, fill opacity=0.2, thick, rounded corners, xshift=1.7ex] (2.5, -4.5) -- (9, -4.5) -- (9,-7.5) -- (2.5,-7.5) -- cycle;
\filldraw[color=green, fill opacity=0.2, thick, rounded corners, xshift=1.7ex] (-2.25, 0.5) -- (5.25, 0.5) -- (5.25,-2.5) -- (-2.25,-2.5) -- cycle;
\node at (0,0) {$\upi_0 \Sigma^{-4\orho} \KH \ul{B}(2,0) \oplus \ulg^2$};
\node at (4,0) {$\upi_0 \Sigma^{-4\orho} \KHN$};
\node at (8,0) {0};
\node at (0,-1) {$\upi_{-1} \Sigma^{-4\orho} \KH \ul{B}(2,0) \oplus \ulg^2$};
\node at (4,-1) {$\upi_{-1} \Sigma^{-4\orho} \KHN$};
\node at (8,-1) {0};
\node at (0,-2) {$\upi_{-2} \Sigma^{-4\orho} \KH \ul{B}(2,0) \oplus \ulg^2$};
\node at (4,-2) {$\upi_{-2} \Sigma^{-4\orho} \KHN$};
\node at (8,-2) {0};
\node at (0,-3) {$\upi_{-3} \Sigma^{-4\orho} \KH \ul{B}(2,0) \oplus \ulg^2$};
\node at (4,-3) {$\upi_{-3} \Sigma^{-4\orho} \KHN$};
\node at (8,-3) {$\upi_{-3} \Sigma^{-4\orho} \KH \ulF$};
\node at (0,-4) {$\upi_{-4} \Sigma^{-4\orho} \KH \ul{B}(2,0) \oplus \ulg^2$};
\node at (4,-4) {$\upi_{-4} \Sigma^{-4\orho} \KHN$};
\node at (8,-4) {$\upi_{-4} \Sigma^{-4\orho} \KH \ulF$};
\node at (0,-5) {0};
\node at (4,-5) {$\upi_{-5} \Sigma^{-4\orho} \KHN$};
\node at (8,-5) {$\upi_{-5} \Sigma^{-4\orho} \KH \ulF$};
\node at (0,-6) {$\vdots$};
\node at (4,-6) {$\vdots$};
\node at (8,-6) {$\vdots$};
\node at (0,-7) {0};
\node at (4,-7) {$\upi_{-12} \Sigma^{-4\orho} \KHN$};
\node at (8,-7) {$\upi_{-12} \Sigma^{-4\orho} \KH \ulF$};
\end{tikzpicture}
\end{figure}

\begin{prop}
\label{prop:PI<-kHKN}
The map $\Sigma^{-k\orho} \KH N_e^\K \F_2 \to \Sigma^{-k\orho} \KH \ulF$ is an isomorphism on $\upi_n$ for $n < -k$,
while the map $\Sigma^{-k\orho} \KH \ul{B}(2,0) \oplus \ulg^2 \to \Sigma^{-k\orho} \KH N_e^\K \F_2$ is an isomorphism on $\upi_n$ for $n\in \{0,-1,-2\}$.
\end{prop}

\begin{pf}
This follows from the cofiber sequence
\[
	\Sigma^{-k\orho} \KH \ul{B}(2,0) \oplus \ulg^2 \to \Sigma^{-k\orho} \KHN \to \Sigma^{-k\orho} \KH \ulF,
\]
together with the fact that $\Sigma^{-k\orho} \KH \ul{B}(2,0) \oplus \ulg^2$ is $-k$-connective
 and $\Sigma^{-k\orho} \KH \ulF \simeq \Sigma^{-3-(k-1)\orho} \KH \ulF^*$ has no homotopy above dimension $-3$.
\end{pf}

Note that \cref{prop:PI<-kHKN} captures all of the homotopy groups in the case of $-k$ equal to either $-1$ or $-2$. The remaining cases are described in the next result.

\begin{theorem} For $-k\leq -3$,  
the nonzero homotopy Mackey functors of $\Sigma^{-k\orho} \KH N_e^\K \F_2$ are
\[
\ul{\pi}_n \Sigma^{-k\orho} \KH N_e^\K \F_2 \iso
\begin{cases}
\ul{g}^3 & n=0 \\
\ul{g}^2 & n=-1 \\
\ul{g}^3 & n=-2 \\
\ul{g}^{-2n-2} & n\in[-k+1,-3] \\
\phi_{LDR}^* \ulF^* \oplus \ul{g}^{2k-5} & n=-k \\
\ul{\pi}_n \Sigma^{-k\orho} \KH \ul{\bbF}_2 & n \leq -k-1.
\end{cases}
\]
\end{theorem}

\begin{pf}
By \cref{prop:PI<-kHKN}, it remains to capture the homotopy in degrees $n\in [-k,-3]$.
The Mackey functors listed here are the only possibilities that are simultaneously compatible with the fiber sequence 
\[
	\Sigma^{-k\orho} \KH \ul{B}(2,0) \oplus \ulg^2 \to \Sigma^{-k\orho} \KHN \to \Sigma^{-k\orho} \KH \ulF,
\]
as well as the fiber sequence
\[
	\Sigma^{-k\orho} \KH \ulF^* \to \Sigma^{-k\orho} \KHN \to \Sigma^{-k\orho} \KH \ul{E}.
\]
For example, in the case $k=3$, the first fiber sequence provides the short exact sequence 
\[
	\phi_{LDR}^* \ulF^* \into \upi_{-3} \Sigma^{-3\orho} \KHN \onto \ulg,
\]
while the second provides the short exact sequence
\[
	\ulg^3 \into \phi_{LDR}^* \ulF^* \oplus \ulg^4 \onto \upi_{-3} \Sigma^{-3\orho} \KHN.
\]
It follows that $\upi_{-3} \Sigma^{-3\orho} \KHN$ must be $\phi_{LDR}^* \ulF^* \oplus \ulg$.
\end{pf}

\begin{remark}
Dual to the proof of \cref{thm:PInSigmakrhoHKN}, here the map $\upi_n \Sigma^{-k\orho} \KH \ulF^* \to \upi_n \Sigma^{-k\orho} \KH N_e^\K \F_2$ is {\it surjective} for $k<0$.
\end{remark}

The homotopy $\ul{\pi}_{n + k\orho}\KHN = \ul{\pi}_{n + k \orho}\KH N_e^\K \F_2$ is displayed in \cref{fig:HN}.

\section{The Homotopy of \texorpdfstring{$\KH N_D^K \F_2$}{H N_D^K F_2}}
\label{sec:HNDF}

We compute the homotopy Mackey functors $\underline{\pi}^\K_{n+k\orho}(\KHND)$, where $\KND = N_D^\K \underline{\bbF_2}$. The Tambara functor structure of $\KND$ is displayed in \cref{figure:normLtoK}. As in the previous section, we first compute the positive cone.

\subsection{The Positive Cone of \texorpdfstring{$\KH N_D^{\K} \ulF$}{$HN_D$}}
\label{sec:PinkposorhoHND}

We compute the homotopy Mackey functors of the $\K$-spectra $\Sigma^{k\orho} \KHND$. As usual we begin with a short exact sequence of Mackey functors.

Let us write $\overline{\KND}$ for the sub-Mackey functor of $\KND$ generated at the proper subgroups, as displayed in \cref{figure:ReducednormLtoK}. We then have a short exact sequence
\begin{equation}
\label{eq:SESNDK}
	\overline{\KND} \into \KND \onto \ulg.
\end{equation}
Furthermore, we have a short exact sequence of Mackey functors
\begin{equation}
\label{eq:SESNbarDK}
	\ul{v}_D^* \into \overline{\KND}  \onto \varphi^*_D \bbF_2^* \oplus \ul{n}_D^*,
\end{equation}
{where $\ul{v}_D^*$ is as described in \cref{tab-K4Mackey}. }

\begin{figure}
\caption{The $\K$-Mackey functor $\overline{\KND} =\overline{N_D^\K \underline{\F_2}}$, which is the kernel of the augmentation $N_D^\K(\ulF) \to \ulg$.}
\label{figure:ReducednormLtoK}
\begin{center}
	\begin{tikzpicture}[scale=1]
		\node (K) at (0,6) {$\Z/2 \oplus \Z/2\{\genVR\}$};
		\node (L) at (-3,3.5) {$\Z/4$};
		\node (D) at (0,3) {$\Z/2$};
		\node (R) at (3,3.5) {$\Z/4$};
		\node (e) at (0,0) {$\Z/2$};

		\draw[bend right=10,->] (L) to node[fill=white]{1} (e);
		\draw[bend right=10,->] (R) to node[fill=white]{1} (e);
		\draw[bend right=10,->] (D) to node[fill=white]{1} (e);
		\draw[bend right=10,->] (K) to node[fill=white]{$(2\ 0)$} (L);
		\draw[bend right=10,->] (K) to node[fill=white]{$0$} (D);
		\draw[bend right=10,->] (K) to node[fill=white]{$(2\ 0)$} (R);

		\draw[bend right=10,color=\inductioncolor,->] (e) to node[right] {2} (L);
		\draw[bend right=10,color=\inductioncolor,->] (e) to node[right]{2} (R);
		\draw[bend right=10,color=\inductioncolor,->] (e) to node[right]{0} (D);
		\draw[bend right=10,color=\inductioncolor,->] (L) to node[pos=(0.3),right=(0.5ex)]{$\binom{1}{1}$} (K);
		\draw[bend right=10,->,color=\inductioncolor] (D) to node[near start, right]{$\binom{1}{0}$} (K);
		\draw[bend right=15,color=\inductioncolor,->] (R) to node[right]{$\binom{1}{1}$}(K);

	\end{tikzpicture}
\end{center}
\end{figure}

\begin{prop}
\label{Srhon*}
There is an equivalence $\Sigma^{\orho} \KH \ul{n}_D^* \simeq \Sigma^1 \KH \ul{n}_D$.
\end{prop}

\begin{pf}
First, note that since $\ul{n}_D^*$ restricts to zero on $D$, it follows that the natural map $\KH \ul{n}_D^* \to \Sigma^{\sigma_D} \KH \ul{n}_D^*$ is an equivalence by \cref{cor:sigmaH suspensions of trivial H}. 
The cofiber sequence
\[
\begin{tikzcd}[row sep=3mm]
	\K/L_+ \wedge  \KH \ul{n}_D^* 
		\ar[r] 
		\ar[d,phantom,"\rotatebox{-90}{$\simeq$}" description]
		&
	\KH \ul{n}_D^* 
		\ar[r]
		&
	\Sigma^{\sigma_L} \KH \ul{n}_D^*
		\\
	\KH\left(\uparrow_L^\K\!\!\ul{f}\right)
\end{tikzcd}
\]
allows us compute the homotopy Mackey functors of $\Sigma^{\sigma_L} \KH \ul{n}_D^* \simeq \Sigma^{\sigma_L+\sigma_D} \KH \ul{n}_D^*$; {the nonzero homotopy Mackey functors are $\phi_L^*\ul{f}$ in degree one and $\phi_R^*\ul{f}$ in degree zero. This gives a Postnikov fiber sequence}
\[
	\Sigma^1 \KH \phi_L^* \ulf \to \Sigma^{\sigma_L+\sigma_D} \KH \ul{n}_D^* \to \KH \phi_R^* \ulf.
\]
Suspending by $\sigma_R$ then gives a fiber sequence
\[
\begin{tikzcd}[row sep=3mm]
	\Sigma^1 \KH \phi_L^* \ulf
		\ar[r]
		&
	\Sigma^{\orho} \KH \ul{n}_D^*
		\ar[r]
		&
	\Sigma^{\sigma_R} \KH \phi_R^* \ul{f}
		\ar[d,phantom,"\rotatebox{-90}{$\simeq$}" description]
		\\
		&
		&
	\Sigma^1 \KH \phi_R^* \ulF,
\end{tikzcd}
\]
where we have again used the fact that if the restriction of $X$ to $R$ is contractible, then $\Sigma^{\sigma_R}X \simeq X$ (\cref{cor:sigmaH suspensions of trivial H}).
It follows that $\Sigma^{\orho} \KH \ul{n}_D^*$ has homotopy concentrated in degree 1, with Mackey functor an extension of $\phi_R^* \ulF$ by $\phi_L^* \ulf$. The only potential ambiguity in the extension is the restriction from $K$ to $L$. But as the result must be symmetric in $L$ and $R$, we conclude the restriction to $L$ must be nontrivial since the restriction to $R$ is so. {Hence, this extension must be $\ul{n}_D^*$.}
\end{pf}

\begin{prop} 
\label{Srhon}
For $k\geq 1$, the nonzero homotopy Mackey functors of $\Sigma^{k\orho} \KH \ul{n}_D$ are
\[
	\upi_n \Sigma^{k\orho} \KH\ul{n}_D \iso \begin{cases}
	\phi_{LR}^* \underline{\bbF_2} & n= k \\
	\ulg^2 & n \in [1,k-1] \\
	\ulg & n=0.
	\end{cases}
\]
\end{prop}

\begin{pf}
We have a short exact sequence of Mackey functors $\phi_{LR}^* \ulf \into \ul{n}_D \onto \ulg$. 
Suspending by $\orho$ and use of \cref{BackgroundSrhoHM} give a cofiber sequence
\[
	\Sigma^1 \KH \phi_{LR}^* \ulF \to \Sigma^{\orho} \KH \ul{n} \to \KH \ulg.
\]
The homotopy of the higher $\orho$-suspensions follow by use of \zcref{eq:inflatingsuspensions} from the relevant $C_2$-equivariant computations, as in \cref{fig:piC2F2}.
\end{pf}

{Recall that $\ul{v}_D^*$ fits into a short exact sequence \zcref{eq:SESNbarDK} with middle term $\overline{\KND}$. Understanding its homotopy will be crucial for computing the homotopy of $\overline{\KND}$, which we use in turn to compute the homotopy of $\KND$.}

\begin{prop} 
\label{htpySkrhov*}
The nonzero homotopy Mackey functors of $\Sigma^{\orho} \KH \ul{v}_D^*$ are
\[
	\upi_n \Sigma^{\orho} \KH\ul{v}_D^* \iso \begin{cases}
	\ul{\F_2} & n=3 \\
	\phi_D^* \ul{f}  & n=2. \\
	\end{cases}
\]
For $k\geq 2$, the nonzero homotopy Mackey functors of $\Sigma^{k\orho} \KH \ul{v}_D^*$ are
\[
	\upi_n \Sigma^{k\orho} \KH\ul{v}_D^* \iso \begin{cases}
	\upi_n \Sigma^{k\orho} \KH \bbF_2 & n \geq k+2 \\
	\phi_D^* \ul{\bbF_2} \oplus \ulg^{2k-3} & n=k+1 \\
	\ulg^{2n-4} & n \in [3,k]. \\
	\end{cases}
\]
\end{prop}

\begin{pf}
We have a short exact sequence $\phi_D^* \ulF^* \into \ulF^* \onto \ul{v}_D^*$, which gives rise to a cofiber sequence
\[
\begin{tikzcd}[row sep=3mm]
	\Sigma^{\orho} \KH \phi_D^* \ulF^* 
		\ar[r] 
		&  
	\Sigma^{\orho} \KH \ulF^* 
		\ar[r] 
		& 
	\Sigma^{\orho} \KH \ul{v}_D^*. 
		\\
	\Sigma^1 \KH \phi_D^* \ulf 
		\ar[u,phantom,"\rotatebox{-90}{$\simeq$}" description] 
		& 
	\Sigma^3 \KH \ulF 
		\ar[u,phantom,"\rotatebox{-90}{$\simeq$}" description]
\end{tikzcd}
\]
We may rotate this and suspend further to get a cofiber sequence
\[
\Sigma^{3+(k-1)\orho} \KH \ulF \to \Sigma^{k\orho} \KH \ul{v}_D^* \to \Sigma^{2+(k-2)\orho}\KH \phi_D^* \ulF.
\]
{The homotopy Mackey functors of suspensions of $\KH \ulF$ can be read off from \cref{fig:HF}, and the homotopy Mackey functors of $\Sigma^{2+(k-2)\orho}\KH \phi_D^* \ulF$ follow from \zcref{eq:inflatingsuspensions} and \cref{fig:piC2F2}. }

On the other hand, we also have a short exact sequence $\phi_{LR}^* \ulf \into \ul{v}_D^* \onto \ulf$ which gives a cofiber sequence
\[
\begin{tikzcd}[row sep=3mm]
	\Sigma^{k\orho} \KH \phi_{LR}^* \ulf \ar[r] &  \Sigma^{k\orho} \KH \ul{v}_D^* \ar[r] &  \Sigma^{k\orho} \KH \ulf. 
	 \\
	 \Sigma^{1+(k-1)\orho} \KH \phi_{LR}^* \ulF
	 \ar[u,phantom,"\rotatebox{-90}{$\simeq$}" description] 
\end{tikzcd}
\]
The homotopy Mackey functors of the right term {of this last cofiber sequence} are computed in \cite{GY}*{Corollary~7.5}, {and the homotopy Mackey functors of $\Sigma^{1+(k-1)\orho} \KH \phi_{LR}^* \ulF$ follow from \zcref{eq:inflatingsuspensions} and \cref{fig:piC2F2}}. The stated homotopy Mackey functors {of $\Sigma^{k \orho} \KH \ul{v}_D^*$} are the only ones compatible with both long exact sequences and \cref{DVanishesImplies2Torsion}.
\end{pf}

\begin{prop}
\label{HtpySrhoNDK}
The nonzero homotopy Mackey functors of $\Sigma^{\orho} \KH N_D^\K \ulF$ are
\[
	\upi_n \Sigma^{\orho} \KH N_D^\K \ulF \iso
	\begin{cases}
	\underline{\bbF_2} & n =3 \\
	\phi_D^* \ul{f} & n=2 \\
	\phi_{D}^* \ul{f} \oplus \ul{n}_D & n=1 \\
	\ulg & n=0.
	\end{cases}
\]
\end{prop}

\begin{pf}
Suspending \eqref{eq:SESNDK} gives a cofiber sequence 
\[
\begin{tikzcd}[row sep=3mm]
	\Sigma^{\orho} \KH \overline{\KND}
		\ar[r]
		&
	\Sigma^{\orho} \KHND 
		\ar[r]
		&
	\Sigma^{\orho} \KH \ulg 
		\\
		&
		&
	\KH \ulg.
		\ar[u,phantom,"\rotatebox{-90}{$\simeq$}" description]
\end{tikzcd}
\]
It therefore suffices to compute the homotopy Mackey functors of $\Sigma^{\orho} \KH \overline{\KND}$, for which we use the following suspension of \eqref{eq:SESNbarDK}:
\[
\Sigma^{\orho} \KH \ul{v}_D^* \to \Sigma^{\orho} \KH \overline{\KND} \to \Sigma^{\orho} \KH \big( \phi_D^* \underline{\F_2}^* \oplus \ul{n}_D^* \big).
\]
The result now follows by combining \Cref{htpySkrhov*,Srhon*,BackgroundSrhoHM}.
\end{pf}

The homotopy of the higher suspensions is as follows.

\begin{theorem} For $k\geq 2$, the nonzero homotopy Mackey functors of $\Sigma^{k\orho} \KH N_D^\K \ulF$ are
\[
	\upi_n \Sigma^{k\orho} \KH N_D^\K \ulF \iso
	\begin{cases}
	\upi_n \Sigma^{k\orho} \KH \underline{\bbF_2} & n \geq k+2 \\
	\phi_D^* \underline{\bbF_2} \oplus \ulg^{2k-3} & n=k+1 \\
	\phi_{LDR}^* \underline{\bbF_2} \oplus \ulg^{2k-4} & n=k \\
	\ulg^{2n-1} & n \in [2,k-1] \\
	\ulg & n \in [0,1].
	\end{cases}
\]
\end{theorem}

\begin{pf}
The argument is the same as for \cref{HtpySrhoNDK}, using 
the cofiber sequence
\[
\Sigma^{k\orho} \KH \ul{v}_D^* \to \Sigma^{k\orho} \KH \overline{\KND} \to \Sigma^{k\orho} \KH \big( \phi_D^* \underline{\F_2}^* \oplus \ul{n}_D^* \big).
\]
for $k\geq 2$. The left term has homotopy in degrees at least 3, while the right term has homotopy in degrees at most $k$. Therefore the homotopy Mackey functors in degrees greater than $k$ follow from \cref{htpySkrhov*}. And the homotopy below degree 3 follows from \cref{Srhon}.

For the Mackey functors in degrees $k$ and lower, we employ a similar argument to that used in the proof of \cref{thm:PInSigmakrhoHKN}. Consider, for instance, the case $k=4$, where the Mackey functors $\upi_3$ and $\upi_4$ remain to be determined. In the relevant degrees, the long exact sequence arising from the above cofiber sequence takes the form
\[
\begin{tikzcd}[arrows={thick,}]
		& 
	\upi_n \Sigma^{4\orho} \KH \ul{v}_D^* 
		& 
	\upi_n \Sigma^{4\orho} \KH \overline{\KND} 
		& 
	\upi_n \Sigma^{4\orho} \KH (\phi_D^* \ul{\F_2}^* \oplus \ul{n}_D^*) 
		\\
 	n=5 
		\ar[r, phantom, ""{coordinate, name=X}]
		\ar[from=X, to=2-2, dashed]
		& 
	\phi_D^* \ul{\F_2} \oplus \ulg^5  \ar[r,"\cong"] 
		& 
	\phi_D^* \ul{\F_2} \oplus \ulg^5  
		\ar[r, "0"]
		\ar[d, phantom, ""{coordinate,name=A}]
		& 
	0	\ar[dll, rounded corners,
			to path={ -- ([xshift=12mm]\tikztostart.east)
					  |- (A)
					  -| ([xshift=-5mm]\tikztotarget.west)
					  -- (\tikztotarget)}]
		\\
	n=4 
		& 
	\ulg^4 \ar[r,hookrightarrow] 
		& 
	? \ar[r] 
		\ar[d, phantom, ""{coordinate,name=B}]
		& 
	\phi_{D}^* \ul{\F_2} \oplus \phi_{LR}^*\ul{\F_2}
		\ar[dll, rounded corners,
			to path={ -- ([xshift=2mm]\tikztostart.east)
					  |- (B)
					  -| ([xshift=-5mm]\tikztotarget.west)
					  -- (\tikztotarget)}]
		\\
	n=3 
		& 
	\ulg^2 \ar[r]  
		& 
	?? \ar[r,->>] 
		\ar[d,phantom,""{coordinate,name=C}]
		& 
	\ulg \oplus \ulg^2
		\ar[dll, rounded corners,
			to path={ -- ([xshift=9mm]\tikztostart.east)
					  |- (C)
					  -| ([xshift=-5mm]\tikztotarget.west)
					  -- (\tikztotarget)}]
		\\
	n=2 
		& 
	0 
		\ar[r, "0"]
		& 
	\ulg^3 
		\ar[r,"\cong"] 
		& 
	\ulg \oplus \ulg^2
		\ar[r, phantom, ""{coordinate,name=Y}]
		\ar[r, to=Y, dashed]
		&
		\phantom{x}
\end{tikzcd}
\]
On the other hand, the long exact sequence induced by (the rotation of) the cofiber sequence 
for $\ulg \oplus \ul{n}_D \into \KND \onto \ulF$
takes the form
\[
\begin{tikzcd}[arrows={thick,}]
		&
	\upi_n \Sigma^{4\orho} \KHND
		& 
	\upi_n \Sigma^{4\orho} \KH \ulF 
		& 
	\upi_n \Sigma^{1+4\orho} \KH (\ulg \oplus \ul{n}_D) 
		\\
	n=5 
		\ar[r, phantom, ""{coordinate,name=X}]
		\ar[from=X,to=2-2,dashed]
		& 
	\phi_D^* \ulF \oplus \ulg^5 
		\ar[r,hookrightarrow] 
		& 
	\phi_{LDR}^* \ulF \oplus \ulg^5 
		\ar[r,->>] 
		\ar[d, phantom, ""{coordinate, name=A}]
		& 
	\phi_{LR}^* \ulF
		\ar[dll, rounded corners,
			to path = { -- ([xshift=1.6mm]\tikztostart.east)
						|- (A)
						-| ([xshift=-6mm]\tikztotarget.west)
						-- (\tikztotarget)
					}]
		\\
	n=4 
		& 
	? 
		\ar[r,hookrightarrow] 
		&
	\phi_{LDR}^* \ulF \oplus \ulg^6 
		\ar[r] 
		\ar[d, phantom, ""{coordinate, name=B}]
		& 
	\ulg^2
		\ar[dll, rounded corners,
			to path = { -- ([xshift=5mm]\tikztostart.east)
						|- (B)
						-| ([xshift=-5mm]\tikztotarget.west)
						-- (\tikztotarget)
					}]
		\\
	n=3 
		& 
	?? 
		\ar[r] 
		& 
	\ulg^7 
		\ar[r,two heads] 
		\ar[d, phantom, ""{coordinate,name=C}]
		& 
	\ulg^2
		\ar[dll, rounded corners,
			to path = { -- ([xshift=5mm]\tikztostart.east)
						|- (C)
						-| ([xshift=-5mm]\tikztotarget.west)
						-- (\tikztotarget)
					}]
		\\
	n=2 
		& 
	\ulg^3 
		\ar[r,hookrightarrow] 
		& 
	\ulg^5 
		\ar[r,two heads] 
		& 
	\ulg^2
		\ar[r, phantom, ""{coordinate,name=Y}]
		\ar[r,to=Y,dashed]
		&
	\phantom{x}
 \end{tikzcd}
\]
The claimed values for $\upi_4$ and $\upi_3$ are the only possibilities compatible with both long exact sequences. 
{Note that although one sequence computes the homotopy of $\KND$ and the other computes the homotopy of $\overline{\KND}$, these agree in positive degree as in the proof of \cref{HtpySrhoNDK}.}
The argument works just as well for larger values of $k$.
\end{pf}

We display these results in the fourth quadrant of \cref{fig:HND}. {The region of the fourth quadrant of \cref{fig:HND} where the homotopy agrees with $\KH \ulF$ is highlighted in blue.}

\subsection{The Negative Cone of \texorpdfstring{$\KH N_D^{\K} \ulF$}{$HN_D$}}
\label{sec:PinknegorhoHND}

We now compute the negative cone with coefficients in $\KND = N_D^\K \ulF$, which we display in the second quadrant of \cref{fig:HND}.

\begin{prop}
\label{HtpySnegrhoNDK}
The nonzero homotopy Mackey functors of $\Sigma^{-\orho} \KH N_D^\K \ulF$ are
\[
	\upi_{-n} \Sigma^{-\orho} \KH N_D^\K \ulF \iso
	\begin{cases}
	\underline{\bbF_2}^* & n =3 \\
	\ul{n}_D^* & n=1 \\
	\ulg & n=0.
	\end{cases}
\]
\end{prop}

\begin{pf}
The short exact sequence of Mackey functors $\ulg \oplus \ul{n}_D \into \KND \onto \ulF$ gives a cofiber sequence 
\[
\begin{tikzcd}[row sep=3mm]
	\Sigma^{-\orho} \KH \big(\ulg \oplus \ul{n}_D\big) 
		\ar[r] 
		&  
	\Sigma^{-\orho} \KHND 
		\ar[r] 
		& 
	\Sigma^{-\orho} \KH \ulF 
		\\
	\KH \ulg \vee \Sigma^{-1} \KH \ul{n}_D^*  
		\ar[u,phantom,"\rotatebox{-90}{$\simeq$}" description]	
		& 
		& 
	\Sigma^{-3} \KH \ulF^* 
		\ar[u,phantom,"\rotatebox{-90}{$\simeq$}" description]
\end{tikzcd}
\]
according to \cref{Srhon*}. {The result follows from the associated long exact sequence.}
\end{pf}

\begin{theorem} For $k\geq 2$, the nonzero homotopy Mackey functors of $\Sigma^{-k\orho} \KH N_D^\K \ulF$ are
\[
	\upi_{-n} \Sigma^{-k\orho} \KH N_D^\K \ulF \iso
	\begin{cases}
	\upi_{-n} \Sigma^{-k\orho} \KH \underline{\bbF_2} & n \geq k+1 \\
	\phi_{LR}^* \underline{\bbF_2}^* \oplus \ulg^{2k-5} & n=k \\
	\ulg^{2n-3} & n \in [3,k-1] \\
	\ulg^2 & n = 2 \\
	\ulg & n = 0,1.
	\end{cases}
\]
\end{theorem}

\begin{pf}
As in the proof of \cref{HtpySnegrhoNDK}, we have a cofiber sequence 
\[
\begin{tikzcd}[row sep=3mm]
	\Sigma^{-k\orho} \KH \big(\ulg \oplus \ul{n}_D\big) 
		\ar[r] 
		&  
	\Sigma^{-k\orho} \KHND 
		\ar[r] 
		& 
	\Sigma^{-k\orho} \KH \ulF 
		\\
	\KH \ulg \vee \Sigma^{-1-(k-1)\orho} \KH \ul{n}_D^*
		\ar[u,phantom,"\rotatebox{-90}{$\simeq$}" description] 
		& 
		& 
	\Sigma^{-3-(k-1)\orho} \KH \ulF^* 
		\ar[u,phantom,"\rotatebox{-90}{$\simeq$}" description]
\end{tikzcd}
\]
The homotopy Mackey functors of $\Sigma^{-k\orho} \KH \ul{n}_D^*$ are the duals of those given in \cref{Srhon} by Brown-Comenetz duality. Thus $\Sigma^{-k\orho} \KH (\ulg \oplus \ul{n}_D)$ is $(-k)$-connective, so that the homotopy Mackey functors of $\Sigma^{-k\orho} \KHND$ agree with those of $\Sigma^{-k\orho} \KH \ulF$ below degree $-k$. Similarly, $\Sigma^{-k\orho} \KH \ulF$ is $(-3)$-coconnective, so that the homotopy Mackey functors of $\Sigma^{-k\orho} \KHND$ follow in degree $-2$ or higher. The intermediate homotopy Mackey functors are the only ones compatible with \cref{DVanishesImplies2Torsion} and the long exact sequences arising from the above cofiber sequence as well as the cofiber sequence
$\Sigma^{-k\orho}\KH \big(\ulg \oplus \phi_D^* \ulF^*\big) \to \Sigma^{-k\orho} \KHN \to \Sigma^{-k\orho} \KHND.$
\end{pf}

{We display these results in the second quadrant of \cref{fig:HND}. The region of the second quadrant of \cref{fig:HND} where the homotopy agrees with $\KH \ulF$ is highlighted in red.}

\section{Multiplicative Structure}
\label{sec:mult}

We briefly describe some of the multiplicative structure in the bigraded Green functors $\upi_{\blacklozenge} \KHN$ and $\upi_{\blacklozenge} \KHND$.

The $RO(\K)$-graded Green ring $\pi_\bigstar^\K \KH \ulF$ was described in \cite{EB} and \cite{HS}. 
The portion graded by honest, as opposed to virtual, representations
\footnote{This portion is called the ``positive cone'' in \cite{EB}, which is slightly different from our usage of the term in this article.}
is described in \cite{EB}*{Theorem~4.14} as
\[
	\pi_{\bigstar\geq0}^\K \KH \ulF \iso \frac{\F_2[a_L,a_D,a_R,t_L,t_D,t_R]}{a_Lt_Dt_R + t_L a_D t_R + t_L t_D a_R}
\]
where $a_H$ is the Euler class for $\sigma_H$, in degree $-\sigma_H$, and $t_H$ is the orientation class, in degree $1-\sigma_H$.
We have chosen to focus on the $\mathrm{Aut}(\K)$-invariant subring $\pi_{\blacklozenge\geq0}^\K \KH \ulF$ and similarly for our other Eilenberg--Mac~Lane spectra. It follows that this bigraded ring can be described as 
\[
	\pi_{\blacklozenge\geq0}^\K \KH \ulF \iso {\F_2[a,\xL,\xD,\xR,\vL,\vR,\uu]}/{I},
\]
where the generators are
\[
	\xL = t_L a_D a_R, \qquad \xD = a_L t_D a_R, \qquad \xR = a_L a_D t_R,
\]
\[
	\vL = a_L t_D t_R, \qquad \vR = t_L t_D a_R, \qquad \uu = t_L t_D t_R,
\]
and $I$ is the ideal generated by
\[ \begin{split}
\xL\xD + a \vR, \qquad \xD\xR + a \vL, \qquad \xL \xR + a\vL + a \vR, \\
\xL\vL + a \uu, \qquad \vR \xR + a \uu, \qquad \vL \xD + \xD \vR + a \uu, \\
\vL^2 + \xD \uu + \xR \uu, \qquad \vL \vR + \xD \uu, \qquad \vR^2 + \xL \uu + \xD \uu.
\end{split}\]
The negative cone is more poorly behaved. However, it contains a class $\Theta$ in degree $-3+\orho$ that is infinitely divisible by each of the {multiplicative} generators of the positive cone.

In \cref{fig:HFmult}, vertical (purple) 
lines indicate multiplications by $a=a_La_Da_R$. 
We use rainbows to indicate multiplication by $\xL$, $\xD$, and $\xR$, {though we omit the subscripts from the generator names} (as indicated in the Key), in order to avoid clutter in the figures. Thus the rainbow connecting the pentagon in (1,-1) to the pentagon in (2,-2) and the vertical line from (2,-1) to (2,-2) indicate that basis elements in (2,-2) are $\xL^2$, $\xD^2$,  $\xR^2$, $a\vL$, and $a\vR$. On the other hand, the rainbow from (2,-1) to (3,-2) indicates that a basis in (3,-2) is given by $\xL\vR$, $\vL\xR$,  $\xD\vL$, and $a\uu$. Note that $\xD\vR$ is equal to $\xD\vL + a\uu$, and that $\xL\vL$ and $\xR\vR$ are both equal to $a\uu$, though this is not indicated in the figure. 

Thus, our use of multiplication lines is to indicate choices of basis elements, rather than to display all possible multiplications. 
For example, the element $a$ in (0,-1) also supports a rainbow, {though} we have not included it in the figure.  

As indicated in \cref{fig:HNmult}, the multiplicative structure of $\upi^\K_{\blacklozenge} \KHN$ is largely the same as that of $\upi^\K_{\blacklozenge} \KH \ulF$. Important differences include:
\begin{enumerate}
\item the absence of $\vL$ and $\vR$ and their corresponding rainbow
\item  the absence of $\xL$, $\xD$, and $\xR$, though their restrictions appear
\item the elements $4$, $\genUL$, and $\genUR$ in $\pi_0^\K$ are infinitely $a$-divisible
\item the rainbows indicate multiplication by the generators of $\pi_{2-2\orho} \KHN$. These generators correspond to the elements $\xL^2$, $\xD^2$, and $\xR^2$ in $\pi_{2-2\orho} \KH \ulF$, up to $a$-multiples. 
\end{enumerate}
The multiplicative structure of $\upi^\K_{\blacklozenge} \KHND$ indicated in \cref{fig:HNDmult} is intermediate between that of $\upi^\K_\blacklozenge \KHN$ and $\upi^\K_\blacklozenge \KH$. For instance, the element $\xL+\xR$ is present in $\pi_{1-\orho} \KHND$ but not in $\pi_{1-\orho} \KHN$.
Here, the yellow arcs in the diagonal $x+y=1$ denote multiplication by the corresponding generator in $\pi_{2-2\orho} \KHND$; as discussed above, this is $\xD^2$, modulo $a$. However, we again warn the reader that many multiplications are not indicated in \cref{fig:HNDmult}. For example, in $\pi^\K_\blacklozenge \KH \ulF$, we have the relation $\xL^2 \cdot \xD\vL = a^2 \vR \uu$, which gives rise to an analogous formula in $\pi^\K_\blacklozenge \KHND$.

\begin{remark}
We have here discussed the multiplicative structure on $\upi_{\blacklozenge} \KHN$ and $\upi_{\blacklozenge} \KHND$. In other words, we have described the graded Green functors. These have more structure: they are graded Tambara functors. However, here there is not much additional data carried in the norms. In the case of $\upi_{\blacklozenge} \KHN$, we have a norm
\[
	\nm_e^\K \colon \pi_n H\F_2 \to \pi_{n+n\orho} \KHN.
\]
This can only be nonzero when $n$ is equal to zero, in which case this norm has been described in \cref{sec:NeKF}. Similarly, for the intermediate norm 
\[
	\nm_L^\K\colon \pi_{n+k\sigma} \CH \ulF \to \pi_{n+n\sigma_L + k\sigma_D + k\sigma_R} \KHN
\]
to land in the $\blacklozenge$-grading, one must have $n=k$. Then the source group is only nonzero when $n$ is equal to zero,  in which case this norm has been described in \cref{sec:NeKF}.
\end{remark}


\section{Tables and Charts}
\label{sec:charts}

Here, we display charts of homotopy Mackey functors for $C_2$-equivariant and $\K$-equivariant Eilenberg--Mac Lane spectra, including our main computations from \cref{sec:HNeF,sec:HNDF}.

We first display $C_2$-equivariant charts.
\cref{tab-C2Mackey} is useful for reading \cref{fig:piC2F2}, \cref{fig:piC2NF2}, and \cref{fig:piC2NHF2}.
\cref{fig:piC2NHF2} was obtained from the work of \cite{MSZ}. In more detail, the homotopy Mackey functors $\upi_n N_e^{C_2} H \bbF_2$ are described in \cite{MSZ}*{Theorem~4.4} for $n\leq 8$.
The data presented in \cref{fig:piC2NHF2} was then deduced from the computation of \cite{MSZ} 
by the facts that $N_e^{C_2} H \bbF_2$ is connective and that its geometric fixed points are  $\Phi^{C_2} N_e^{C_2} H \bbF_2 \simeq H\bbF_2$. The main mechanism used in this process is the long exact sequence
\[
\cdots \to \, \uparrow_e^{C_2}\! \pi_{\dim V} \left( H\F_2 \wedge H\F_2 \right) \to \upi_V N_e^{C_2} H \F_2 \xrightarrow{a} \upi_{V-\sigma} N_e^{C_2} H \F_2 \to \cdots
\]
The shaded region of \cref{fig:piC2NHF2} has not been computed.

\cref{tab-K4Mackey} is useful for reading the $\K$-equivariant charts \cref{fig:HF,fig:HFmult,fig:HN,fig:HNmult,fig:HND,fig:HNDmult}.
As discussed in \cref{sec:HNeF,sec:HNDF}, the shaded regions in \cref{fig:HN,fig:HND} indicate where those charts agree with the previously known \cref{fig:HF}.
The charts \cref{fig:HFmult,fig:HNmult,fig:HNDmult} indicate multiplicative structure. This is described in \cref{sec:mult}.

\clearpage


\begin{table}[h] 
\caption[Some  {$C_{2}$}-Mackey functors]{Some $C_{2}$-Mackey functors}
\label{tab-C2Mackey}
\begin{center}
{\renewcommand{\arraystretch}{1.2}
\begin{tabular}{|c|c|c|c|c|c|}
         \hline
         $\bullet=\ulg=\phi_{C_2}^* \F_2 $
         &  $\overline{\bullet}=\ulf$ 
         & $\fillsquare=\ulF$
         & $\fillsquaredual=\ulF^*$
         & $\raisebox{-2pt}{\begin{tikzpicture}\node[draw,fill=white,circle,inner sep=1pt]  at (0,0) {\tiny $\CN$};\end{tikzpicture}} = 
N_e^{C_2} \F_2$
         & $\hat\bullet = \uparrow_e^{C_2} \!\F_2$
         \\ \hline
          \begin{tikzcd}
            \F_2 \\ 0 
          \end{tikzcd}
         &  
         \begin{tikzcd}
            0 \\ \F_2 
          \end{tikzcd}
         &    
         \begin{tikzcd}
            \F_2 \ar[d,"1" swap] \\ \F_2
          \end{tikzcd}
         & 
          \begin{tikzcd}
            \F_2  \\ \F_2 \ar[u,"1" swap]  
          \end{tikzcd}
          &
          \begin{tikzcd}
          \Z/4 \ar[d,bend right, "1" swap] \\ \Z/2 \ar[u,bend right, "2" swap]
          \end{tikzcd}
          &
          \begin{tikzcd}
          \F_2 \ar[d,bend right, "\Delta" swap] \\ \F_2\{C_2/e\}
          \ar[u,bend right, "\nabla" swap]
          \end{tikzcd}
\\
\hline
\end{tabular} }
\end{center}
\end{table}


%
%

\begin{minipage}{\linewidth}
      \centering
      \begin{minipage}{0.475\linewidth}
          \begin{figure}[H]
			\caption{$\upi^{C_2}_{x+y\orho} \CH \ulF$}
			\label{fig:piC2F2}
              \includegraphics[width=\linewidth]{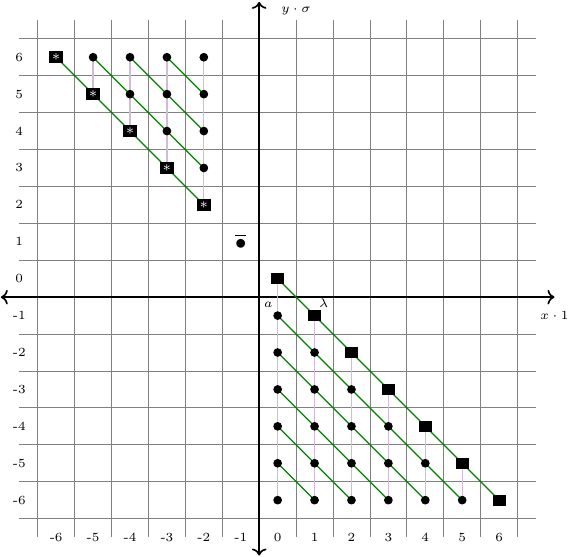}
          \end{figure}
      \end{minipage}
      \hspace{0.025\linewidth}
      \begin{minipage}{0.475\linewidth}
          \begin{figure}[H]
			\caption{$\upi^{C_2}_{x+y\orho} \CH N_e^{C_2} \F_2$}
			\label{fig:piC2NF2}
              \includegraphics[width=\linewidth]{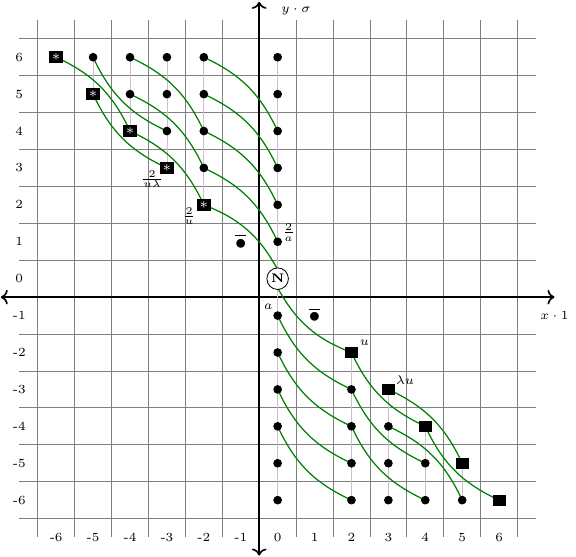}
          \end{figure}
      \end{minipage}

      \begin{minipage}{0.475\linewidth}
\begin{figure}[H]
\caption{$\upi^{C_2}_{x+y\orho} N_e^{C_2} H\F_2$}
\label{fig:piC2NHF2}
\includegraphics[width=\linewidth]{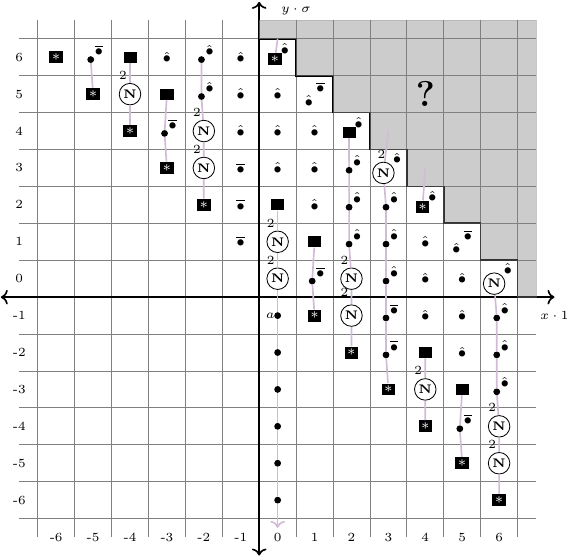}
\end{figure}
      \end{minipage}
  \end{minipage}


\begin{table}[h]
\caption{Some $K_{4}$-Mackey functors}
\label{tab-K4Mackey}
\begin{adjustbox}{width=0.9\columnwidth,center}
\begin{tabular}{|c|c|c|}
         \hline
        $\blacksquare=\ulF$ 
        &
         $\abox=\ulF^*$
         &
         $\circ=\ul B(2,0) $
         \\
         \hline
         \begin{tikzcd}[ swap]
         & \F_2 \ar[dl,"1"] \ar[d,"1"] \ar[dr,"1" swap]  & \\
         \F_2 \ar[dr,"1" pos=0.3] & \F_2 \ar[d,"1"]  & \F_2 \ar[dl,"1" swap, pos=0.4] \\
          & \F_2  
         \end{tikzcd}
         &
         \begin{tikzcd}[swap]
         & \F_2  & \\
         \F_2  \ar[ur,"1" swap]  & \F_2 \ar[u,"1"]  & \F_2 \ar[ul,"1"]  \\
          & \F_2  \ar[ul,"1" swap] \ar[u,"1"] \ar[ur,"1"] 
          \end{tikzcd}
         &
         \begin{tikzcd}[bend right=8, swap]
         & \Z/4 \ar[dl,"1"] \ar[d,"1"] \ar[dr,"1"] & \\
         \Z/2 \ar[ur,"2"]  & \Z/2 \ar[u,"2"]  & \Z/2 \ar[ul,"2"]  \\
          & 0 
         \end{tikzcd}
         \\ 
         \hline 
         $\begin{tikzpicture}
      \node[regular polygon, fill=black, draw, regular polygon sides=5, 
 minimum width=0pt, inner sep = 0.5ex,] at (0,0) {};
    \end{tikzpicture} =
\phi^*_{LDR} (\ulF)$
         &
         $\raisebox{-3pt}{\apent} = \phi^*_{LDR}(\ulF^*)$
         &
         $\phiLDRf=\phi^*_{LDR}(\ul f)$
         \\ \hline
         \begin{tikzcd}[ swap]
         & \F^3_2 \ar[dl,"p_1"] \ar[d,"p_2"] \ar[dr,"p_3" swap]  & \\
         \F_2  & \F_2  & \F_2 \\
          & 0  
         \end{tikzcd}
         &
         \begin{tikzcd}[]
         & \F^3_2    & \\
         \F_2 \ar[ur,"\iota_1"]  & \F_2 \ar[u,"\iota_2"'] & \F_2 \ar[ul,"\iota_3" swap] \\
          & 0  
         \end{tikzcd}
         &
         \begin{tikzcd}[ swap]
         & 0 & \\
         \F_2  & \F_2  & \F_2 \\
          & 0  
         \end{tikzcd}
         \\ 
         \hline 
         $\begin{tikzpicture}
      \node[trapezium, fill=black, draw, 
 minimum width=0pt, inner sep = 0.5ex,] at (0,0) {};
    \end{tikzpicture} = \ul{mg}$
         &
         $\raisebox{-3pt}{\begin{tikzpicture}
      \node[trapezium, fill={gray!50}, draw, 
 minimum width=0pt, inner sep = 0.5ex,] at (0,0) {};
 \node at (0,0) {$\ast$};
    \end{tikzpicture}} = \ul{mg}^*$
         &
         $\bullet=\ulg$
         \\ \hline
         \begin{tikzcd}[ swap]
         & \F^2_2 \ar[dl,"p_1"] \ar[d,"\nabla"] \ar[dr,"p_2" swap]  & \\
         \F_2  & \F_2 & \F_2 \\
          & 0  
         \end{tikzcd}
         &
         \begin{tikzcd}[]
         & \F^2_2     & \\
         \F_2\ar[ur, "\iota_1"]  & \F_2 \ar[u,"\Delta"'] & \F_2 \ar[ul,"\iota_2" swap] \\
          & 0  
         \end{tikzcd}
         &
         \begin{tikzcd}[swap]
         & \F_2 & \\
         0  & 0  & 0 \\
          & 0  
         \end{tikzcd}
         \\ 
         \hline 
         $\nH = \ul{n}_D$
         &
         $\nHdual = \ul{n}_D^*$
         &
         $\ul{v}_D^*$
         \\ \hline
         \begin{tikzcd}[ swap]
         & \F_2 \ar[dl,"1"]  \ar[dr,"1" swap]  & \\
         \F_2  & 0  & \F_2 \\
          & 0  
         \end{tikzcd}
         &
         \begin{tikzcd}[]
         & \F_2    & \\
         \F_2 \ar[ur,"1"]  & 0 & \F_2 \ar[ul,"1" swap] \\
          & 0  
         \end{tikzcd}
         &
         \begin{tikzcd}[]
         & 0    & \\
         \F_2   & 0 & \F_2  \\
          & \F_2 \ar[ul,"1" ] \ar[ur,"1" swap]
         \end{tikzcd}
\\
\hline
\end{tabular} 
\end{adjustbox}
\end{table}

\clearpage


\begin{figure}
\caption{$\upi^\K_{x+y\orho} \KH \ulF$}
\label{fig:HF}
\includegraphics[width=\textwidth]{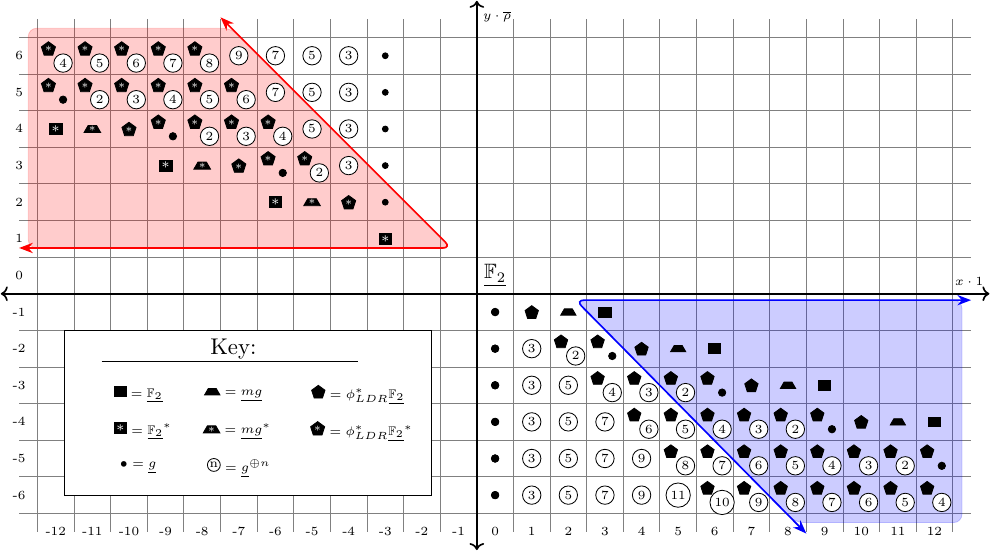}
\end{figure}

\begin{figure}
\caption{$\upi^\K_{x+y\orho} \KH \ulF$, with multiplicative structure emphasized}
\label{fig:HFmult}
\includegraphics[width=\textwidth]{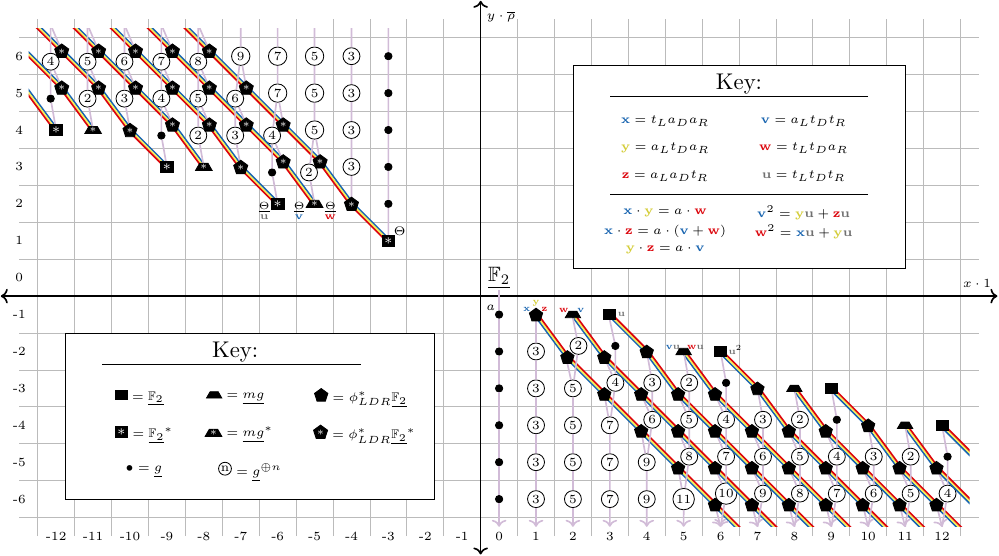}
\end{figure}

\begin{figure}
\caption{$\upi^\K_{x+y\orho} \KH N_e^\K \bbF_2$}
\label{fig:HN}
\includegraphics[width=\textwidth]{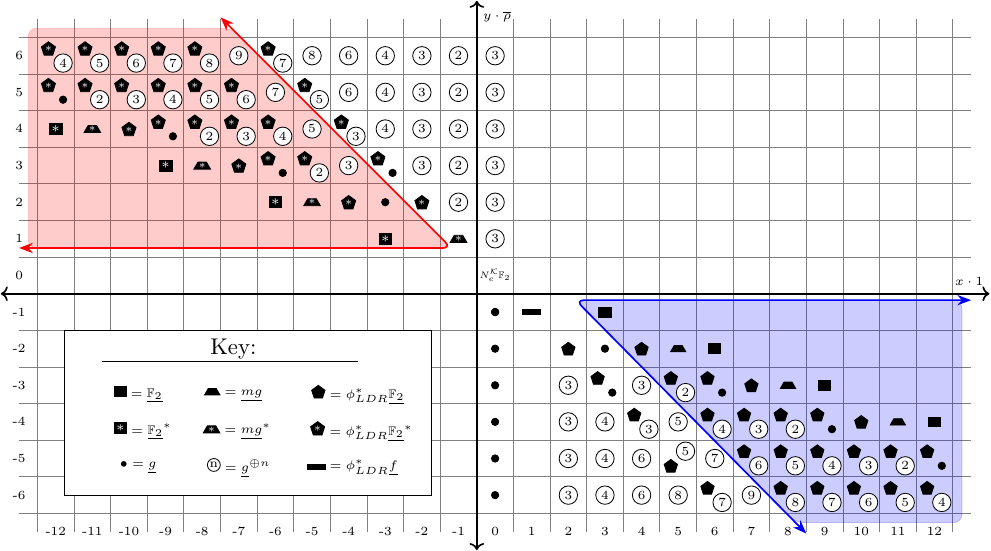}
\end{figure}

\begin{figure}
\caption{$\upi^\K_{x+y\orho} \KH N_e^\K \bbF_2$, with multiplicative structure emphasized}
\label{fig:HNmult}
\includegraphics[width=\textwidth]{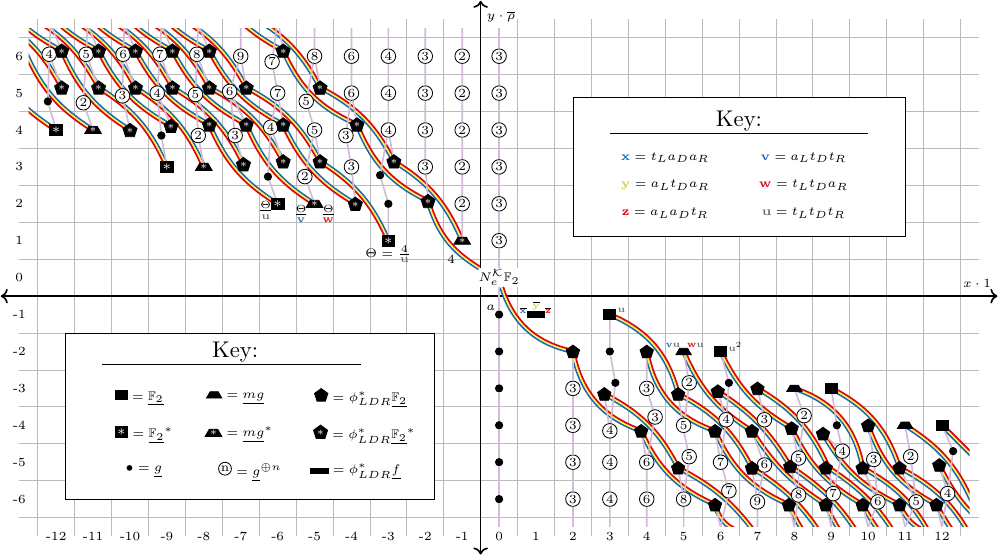}
\end{figure}

\begin{figure}
\caption{$\upi^\K_{x+y\orho} \KH N_D^\K \ulF$}
\label{fig:HND}
\includegraphics[width=\textwidth]{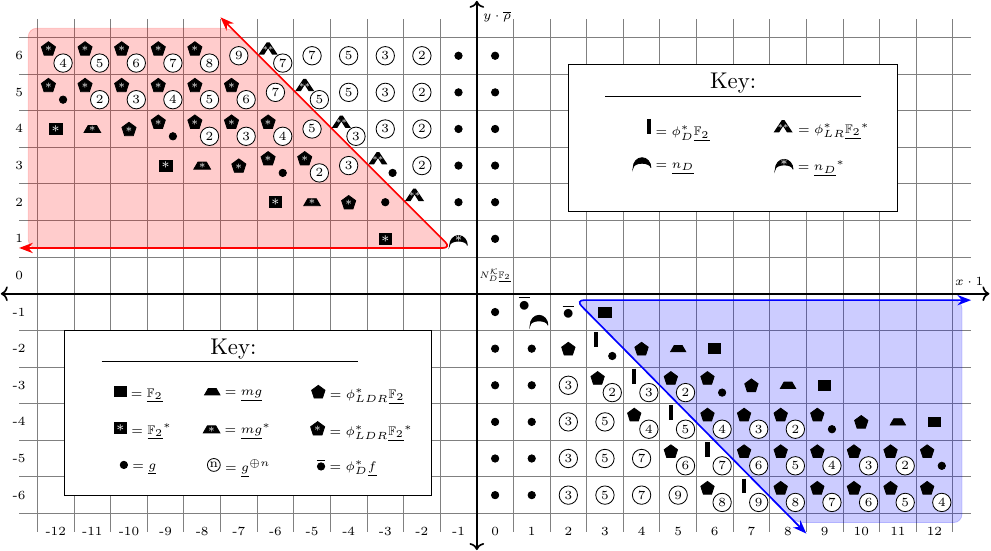}
\end{figure}

\begin{figure}
\caption{$\upi^\K_{x+y\orho} \KH N_D^\K \bbF_2$, with multiplicative structure emphasized}
\label{fig:HNDmult}
\includegraphics[width=\textwidth]{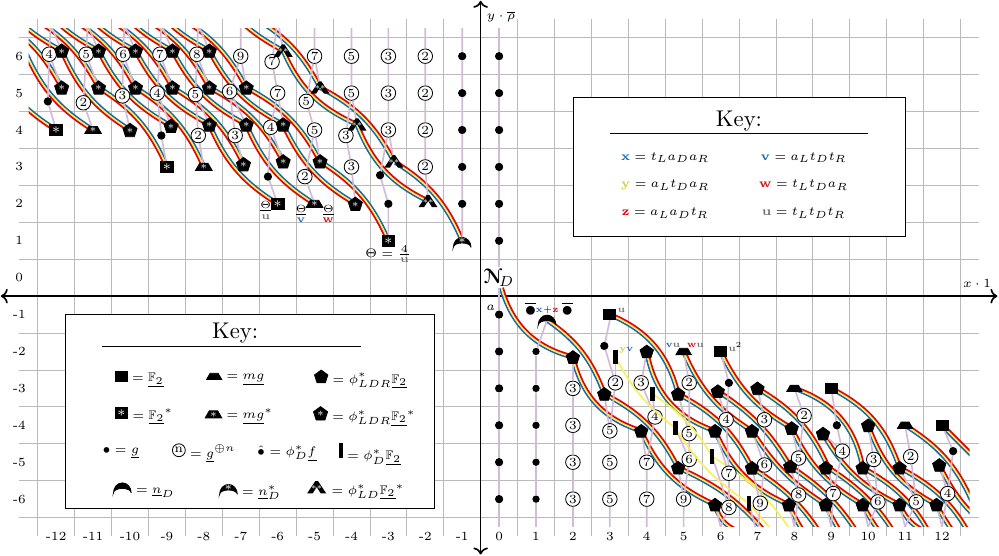}
\end{figure}

\clearpage

\bibliographystyle{amsalpha}

\begin{bibdiv}
\begin{biblist}

\end{biblist}
\end{bibdiv}


\begin{bibdiv}
\begin{biblist}

\bib{BGHL}{article}{
   author={Blumberg, Andrew J.},
   author={Gerhardt, Teena},
   author={Hill, Michael A.},
   author={Lawson, Tyler},
   title={The Witt vectors for Green functors},
   journal={J. Algebra},
   volume={537},
   date={2019},
   pages={197--244},
   issn={0021-8693},
   review={\MR{3990042}},
   doi={10.1016/j.jalgebra.2019.07.014},
}

\bib{BH2018}{article}{
	AUTHOR = {Blumberg, Andrew J.},
	AUTHOR = {Hill, Michael A.},
     TITLE = {Incomplete {T}ambara functors},
   JOURNAL = {Algebr. Geom. Topol.},
  FJOURNAL = {Algebraic \& Geometric Topology},
    VOLUME = {18},
      YEAR = {2018},
    NUMBER = {2},
     PAGES = {723--766},
      ISSN = {1472-2747,1472-2739},
}

\bib{Dugger}{article}{
   author={Dugger, Daniel},
   title={An Atiyah-Hirzebruch spectral sequence for $KR$-theory},
   journal={$K$-Theory},
   volume={35},
   date={2005},
   number={3-4},
   pages={213--256 (2006)},
   issn={0920-3036},
   review={\MR{2240234}},
   doi={10.1007/s10977-005-1552-9},
}

\bib{EB}{thesis}{
	author={Ellis-Bloor, B.},
	title={The Klein four homotopy Mackey functor structure of {$H\F_2$}},
	note={ANU Honours Thesis},
	date={2020},
	eprint={https://openresearch-repository.anu.edu.au/bitstream/1885/258402/1/Honours_thesis.pdf},
}

\bib{GY}{article}{
   author={Guillou, B.},
   author={Yarnall, C.},
   title={The Klein four slices of ${\Sigma}^nH\underline{\Bbb F}_2$},
   journal={Math. Z.},
   volume={295},
   date={2020},
   number={3-4},
   pages={1405--1441},
   issn={0025-5874},
   review={\MR{4125695}},
   doi={10.1007/s00209-019-02433-3},
}

\bib{HS}{article}{
	author={Hausmann, M.},
	author={Schwede, S.},
	title={Representation-graded Bredon homology of elementary abelian 2-groups},
	JOURNAL = {Algebr. Geom. Topol.},
  	FJOURNAL = {Algebraic \& Geometric Topology},
	volume={25},
	year={2025},
	number={7},
	pages={4321--4340},
	DOI = {10.2140/agt.2025.25.4321},
}

\bib{Hill2017}{article}{
    AUTHOR = {Hill, Michael A.},
     TITLE = {On the {A}ndr\'e-{Q}uillen homology of {T}ambara functors},
   JOURNAL = {J. Algebra},
  FJOURNAL = {Journal of Algebra},
    VOLUME = {489},
      YEAR = {2017},
     PAGES = {115--137},
      ISSN = {0021-8693,1090-266X},
       DOI = {10.1016/j.jalgebra.2017.06.029},
       URL = {https://doi.org/10.1016/j.jalgebra.2017.06.029},
}

\bib{Slice}{article}{
   author={Hill, Michael A.},
   title={The equivariant slice filtration: a primer},
   journal={Homology Homotopy Appl.},
   volume={14},
   date={2012},
   number={2},
   pages={143--166},
   issn={1532-0073},
   review={\MR{3007090}},
   doi={10.4310/HHA.2012.v14.n2.a9},
}

\bib{Kervaire}{article}{
   author={Hill, M. A.},
   author={Hopkins, M. J.},
   author={Ravenel, D. C.},
   title={On the nonexistence of elements of Kervaire invariant one},
   journal={Ann. of Math. (2)},
   volume={184},
   date={2016},
   number={1},
   pages={1--262},
   issn={0003-486X},
   review={\MR{3505179}},
   doi={10.4007/annals.2016.184.1.1},
}

\bib{HM2019}{article}{
    AUTHOR = {Hill, Michael A.},
    AUTHOR = {Mazur, Kristen},
     TITLE = {An equivariant tensor product on {M}ackey functors},
   JOURNAL = {J. Pure Appl. Algebra},
  FJOURNAL = {Journal of Pure and Applied Algebra},
    VOLUME = {223},
      YEAR = {2019},
    NUMBER = {12},
     PAGES = {5310--5345},
      ISSN = {0022-4049,1873-1376},
       DOI = {10.1016/j.jpaa.2019.04.001},
       URL = {https://doi.org/10.1016/j.jpaa.2019.04.001},
}

\bib{HMQ}{article}{
  author   = {Hill, Michael A.},
  author   = {Mehrle, David},
  author   = {Quigley, James D.},
  title    = {Free incomplete {T}ambara functors are almost never flat},
  journal  = {Int. Math. Res. Not. IMRN},
  fjournal = {International Mathematics Research Notices. IMRN},
  year     = {2023},
  number   = {5},
  pages    = {4225--4291},
  issn     = {1073-7928,1687-0247},
  doi      = {10.1093/imrn/rnab361},
  url      = {https://doi.org/10.1093/imrn/rnab361},
}

\bib{HSWX}{article}{
   author={Hill, Michael A.},
   author={Shi, XiaoLin Danny},
   author={Wang, Guozhen},
   author={Xu, Zhouli},
   title={The slice spectral sequence of a $C_4$-equivariant height-4
   Lubin-Tate theory},
   journal={Mem. Amer. Math. Soc.},
   volume={288},
   date={2023},
   number={1429},
   pages={v+119},
   issn={0065-9266},
   isbn={978-1-4704-7468-3; 978-1-4704-7571-0},
   review={\MR{4627086}},
   doi={10.1090/memo/1429},
}

\bib{HoKr}{article}{
   author={Holler, John},
   author={Kriz, Igor},
   label={HoKr},
   title={On ${\rm RO}(G)$-graded equivariant ``ordinary'' cohomology where
   $G$ is a power of ${\Bbb Z}/2$},
   journal={Algebr. Geom. Topol.},
   volume={17},
   date={2017},
   number={2},
   pages={741--763},
   issn={1472-2747},
   review={\MR{3623670}},
   doi={10.2140/agt.2017.17.741},
}


\bib{Hoyer}{thesis}{
	author={Hoyer, Rolf},
	label={Ho},
	title={Two topics in stable homotopy theory},
	note={Thesis (Ph.D.)--The University of Chicago},
	PUBLISHER = {ProQuest LLC, Ann Arbor, MI},
    YEAR = {2014},
    PAGES = {93},
    ISBN = {978-1321-03338-0},
}

\bib{JK}{article}{
	author={Keyes, J.},
	title={The $RO(K)$-graded coefficients of $H\underline{A}$},
	date={2025},
	eprint={https://arxiv.org/abs/2503.03173},
}

\bib{MSZ}{article}{
   author={Meier, Lennart},
   author={Shi, XiaoLin Danny},
   author={Zeng, Mingcong},
   title={The localized slice spectral sequence, norms of Real bordism, and
   the Segal conjecture},
   journal={Adv. Math.},
   volume={412},
   date={2023},
   pages={Paper No. 108804, 74},
   issn={0001-8708},
   review={\MR{4517348}},
   doi={10.1016/j.aim.2022.108804},
}

\bib{ideals}{article}{
    AUTHOR = {Nakaoka, Hiroyuki},
     TITLE = {Ideals of {T}ambara functors},
   JOURNAL = {Adv. Math.},
  FJOURNAL = {Advances in Mathematics},
    VOLUME = {230},
      YEAR = {2012},
    NUMBER = {4-6},
     PAGES = {2295--2331},
      ISSN = {0001-8708,1090-2082},
       DOI = {10.1016/j.aim.2012.04.021},
       URL = {https://doi.org/10.1016/j.aim.2012.04.021},
}

\bib{Sikora}{article}{
   author={Sikora, Igor},
   label={Si},
   title={On the $RO(Q)$-graded coefficients of Eilenberg--Mac~Lane spectra},
   journal={J. Homotopy Relat. Struct.},
   volume={17},
   date={2022},
   number={4},
   pages={525--568},
   issn={2193-8407},
   review={\MR{4514125}},
   doi={10.1007/s40062-022-00314-x},
}

\bib{Slone}{article}{
   author={Slone, Carissa},
   label={Sl},
   title={Klein four 2-slices and the slices of $\Sigma^{\pm
   n}H\underline{\Bbb Z}$},
   journal={Math. Z.},
   volume={301},
   date={2022},
   number={4},
   pages={3895--3938},
   issn={0025-5874},
   review={\MR{4449734}},
   doi={10.1007/s00209-022-03022-7},
}

\bib{Tam1993}{article}{
  title    = {On Multiplicative Transfer},
  author   = {Tambara, D.},
  year     = {1993},
  month    = {jan},
  volume   = {21},
  pages    = {1393--1420},
  issn     = {0092-7872, 1532-4125},
  doi      = {10.1080/00927879308824627},
  journal  = {Communications in Algebra},
  language = {en},
  number   = {4},
}

\bib{U}{article}{
	author = {Ullman, John},
	title = {Symmetric Powers and Norms of Mackey Functors},
	eprint = {https://arxiv.org/pdf/1304.5648},
	year = {2013},
}


\end{biblist}
\end{bibdiv}

\end{document}